%% file: wanderingarx.tex
\documentclass{amsart}

\usepackage[english]{babel}
\usepackage[utf8]{inputenc}

\usepackage{amscd}
\usepackage{latexsym}
\usepackage{amsfonts}
\usepackage{amssymb}
\usepackage{amsthm}
\usepackage{amsmath}
\usepackage{xypic}

\usepackage{graphicx}
\usepackage{graphics}

\usepackage{multirow}
\usepackage{fancyvrb}
\usepackage{color}

\usepackage{enumerate}
\usepackage{verbatim}
\usepackage{import}

 \usepackage[colorlinks=true]{hyperref}
\hypersetup{urlcolor=blue, citecolor=red}
\usepackage{subfigure}

\usepackage{pgf,tikz}
\usepackage{mathrsfs}
\usetikzlibrary{arrows}

\usepackage{pgfplots}
\tikzset{ font={\fontsize{9pt}{12}\selectfont}}

\newtheorem{theorem}{Theorem}
\newtheorem{proposition}[theorem]{Proposition}
\newtheorem{lemma}[theorem]{Lemma}

\newtheorem{corollary}[theorem]{Corollary}

\theoremstyle{definition}
\newtheorem{definition}[theorem]{Definition}
\newtheorem{remark}[theorem]{Remark}

\newtheorem{example}[theorem]{Example}

\newcommand{\com}{\mathbb{C}}
\newcommand{\C}{\mathbb{C}}
\newcommand{\R}{\mathbb{R}}

\newcommand{\real}{\mathbb{R}}

\newcommand{\Z}{\mathbb{Z}}

\newcommand{\mj}{\mathcal{J}}

\newcommand{\mk}{\mathcal{K}}

\def\s{\mathbb{S}}
\def\cal {\mathcal}
\def\D{{\mathbb D}}
\def\eps{{\epsilon}}
\def\vareps{{\varepsilon}}

\def\bo{{\boldsymbol \omega}}

\def\T{{\mathbb T}}

\def\PP{{\mathfrak A}}
\def\VV{{\mathfrak V}}
\def\UU{{\mathfrak U}}
\def\p{{\mathfrak p}}
\def\q{{\mathfrak q}}

\def\diam{{\rm diam}}
\def\cf{{\cal C}_f}
\def\pf{{\cal P}_f}

\def\a{{\mathfrak a}}
\def\b{{\mathfrak b}}

\title{Julia sets with a wandering branching point}

\begin{author}[X.~Buff]{Xavier Buff}\thanks{This research was supported in part by the ANR grant Lambda ANR-13-BS01-0002}
\email{xavier.buff@math.univ-toulouse.fr}
\address{ %
  Institut de Math\'ematiques de Toulouse\\
 Universit\'e Paul Sabatier\\
  118, route de Narbonne \\
  31062 Toulouse Cedex \\
  France }
\end{author}

\begin{author}[J.~Canela]{Jordi Canela}
\email{Jordi.Canela\_Sanchez@math.univ-toulouse.fr}
\address{ %
Laboratoire d'Analyse et de Mathématiques Apliquées\\
 Université Paris-Est Marne-la-Vallée\\ 
 5, boulevard Descartes\\
 77454 Champs-sur-Marne\\
  France }
\end{author}

\begin{author}[P.~Roesch]{Pascale Roesch}
\email{pascale.roesch@math.univ-toulouse.fr}
\address{ %
  Institut de Math\'ematiques de Toulouse\\
 Universit\'e Paul Sabatier\\
  118, route de Narbonne \\
  31062 Toulouse Cedex \\
  France }
\end{author}

\begin{document}

\begin{abstract}
According to the Thurston No Wandering Triangle Theorem, a branching point in a locally connected quadratic Julia set is either preperiodic or precritical.  Blokh and Oversteegen proved that this theorem does not hold for higher degree Julia sets: there exist cubic polynomials whose Julia set is a locally connected dendrite with a branching point which is neither preperiodic nor precritical. In this article, we reprove this result, constructing such cubic polynomials as limits of cubic polynomials for which one critical point eventually maps to the other critical point which eventually maps to a repelling fixed point. 
\end{abstract}

\maketitle

\section*{Notations}

\begin{itemize}
\item $\C$ is the complex plane, 
\item $\D$ is the unit disk, 
\item $\s^1$ is the unit circle. 
\item $\T:=\R/\Z$. 
\end{itemize}

\section*{Introduction}

In this article, we consider polynomials $f:\C\to \C$ of degree at least $2$ as dynamical systems: the orbit of a point $z\in \C$ is the set 
\[{\cal O}(z):=\{f^{\circ n}(z)\}_{n\geq 0}.\]
The orbit of a point is finite if and only if the point is {\em preperiodic}. More precisely, a point $\alpha\in \C$ is preperiodic if $f^{\circ (r+s)}(\alpha)=f^{\circ r}(\alpha)$ for some integers $r\geq 0$ and $s\geq 1$.  If $r$ and $s$ are minimal integers such that $f^{\circ (r+s)}(\alpha)=f^{\circ r}(\alpha)$, then $r$ is the {\em preperiod} and $s$ is the {\em period}. The point $\alpha$ is {\em periodic} if the preperiod is $0$. In this case, the point is {\em repelling} if $\bigl|(f^{\circ s})'(\alpha)\bigr|>1$.

The filled-in Julia set $\mk_f$ is the set of points with bounded orbit and the Julia set $\mj_f$ is its topological boundary. 
The sets $\mk_f$ and  $\mj_f$ are compact subsets of $\C$. They are completely invariant: $f^{-1}(\mk_f) = f(\mk_f) = \mk_f$ and $f^{-1}(\mj_f) = f(\mj_f) = \mj_f$. 
Preperiodic points are contained in $\mk_f$. Repelling periodic points are contained in $\mj_f$. In fact, $\mj_f$ is the closure of the set of repelling periodic points (see e.g. \cite{Bear, Mi1}). 

A point $\omega\in \C$ is a critical point if the derivative of $f$ vanishes at $\omega$. The topology of $\mk_f$ and $\mj_f$ is related to the behavior of critical orbits. For example, $\mk_f$ and $\mj_f$ are connected if and only if the critical points of $f$ belong to $\mk_f$ (see e.g. \cite{Bear, Mi1}). 

A {\em dendritic polynomial} is a polynomial $f$ for which $\mj_f$ is a {\em dendrite}, i.e., $\mj_f$ is connected and locally connected and contains no simple closed curve. This is the case whenever each critical point is preperiodic to a repelling periodic point (see 
\cite[Th. V.4.2]{CaGa}).  

\begin{example}
The Julia set of the quadratic polynomial $f(z) = z^2+{\rm i}$ is a dendrite (see  Figure \ref{fig:z2i}). The unique critical point is $\omega=0$ and its orbit is 
\[
\xymatrix{0 \ar[r] & {\rm i}  \ar[r] & -1 + {\rm i}\ar@/^0.7pc/[r] & -{\rm i} \ar@/^0.7pc/[l]
}
\]
The derivative of $f^{\circ 2}$ at $-1 + {\rm i}$ is $4+4{\rm i}$ which has modulus $4\sqrt 2>1$. 
\end{example}

\begin{figure}[hbt!]
\centerline{
\includegraphics[height=5cm]{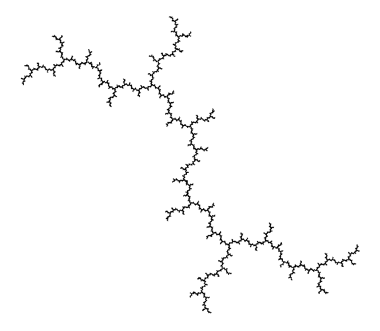}
}
\caption{The Julia set $\mj_f$ for $f(z) = z^2+{\rm i}$ is a dendrite.\label{fig:z2i}
}
\end{figure}

Assume $\mj_f$ is locally connected. 
A point $\xi\in\mj_f$ is a {\em branching point} if $\mj_f\setminus \{\xi\}$ has (at least) three connected components. According to Thurston \cite{Th}, if $f$ is a quadratic polynomial, the orbit of such a branching point contains a periodic point or a critical point. A point $z\in \C$ is {\em precritical} if $f^{\circ n}(z) = \omega$ for some critical point $\omega$ and some integer $n\geq 0$. 

\begin{theorem}[Thurston]
Let $f$ be a quadratic polynomial with locally connected Julia set $\mj_f$. If $\xi\in \mj_f$ is a branching point, then $\xi$ is preperiodic or precritical. 
\end{theorem}

Blokh and Oversteegen \cite{BlOv} proved that such a result does not hold for higher degree polynomials.

\begin{theorem}[Blokh-Oversteegen]\label{thmA}
There exist dendritic cubic polynomials having a branching point  which is neither preperiodic nor precritical.
\end{theorem}

Our goal is to give a new proof of this result. The strategy of our proof consists in exhibiting a sequence $\{f_n\}$ of dendritic polynomials with a sequence $\{\xi_n\}$ of branching points in $\mj_{f_n}$ which are precritical to both critical points of $f_n$ and preperiodic to a repelling fixed point of $f_n$, such that 
\begin{itemize}
\item the sequence $\{f_n\}$ converges to a dendritic polynomial $f$,
\item the sequence $\{\xi_n\}$ converges to a branching point $\xi$ in $\mj_f$ and 
\item $\xi$ is neither precritical nor preperiodic for $f$.  
\end{itemize}
More precisely, we exhibit  sequences satisfying:
\begin{itemize}
\item $f_n^{\circ j_n}(\xi_n) = \omega_n$, $f_n^{\circ k_n}(\omega_n) = \omega'_n$ and $f_n^{\circ \ell_n}(\omega'_n) = \alpha_n$
for some increasing sequences of integers $\{j_n\}$, $\{k_n\}$ and $\{\ell_n\}$, 
where $\omega_n$ and $\omega'_n$ are the two critical points of $f_n$ and $\alpha_n$ is a repelling fixed point of $f_n$; 
\item $\mj_{f_n}\setminus \{\xi_n\}$ has (at least) three connected components containing points $\beta_n$, $\beta'_n$ and $\beta''_n$ such that $f_n(\beta_n) = f_n(\beta'_n)=f_n(\beta''_n)=\beta_n$. 
\end{itemize}
We shall see that we can choose $f_{n+1}$ arbitrarily close to $f_n$ for each $n$; this is enough to deduce Theorem \ref{thmA}. 


The paper is structured as follows. In \S\ref{sec:keyresults} we introduce the notion of admissible polynomials and state key results used in the proof of Theorem~\ref{thmA}: convergence of nodal points, convergence of Carath\'eodory loops and a Key Proposition regarding the existence of a particular sequence of admissible polynomials. In \S\ref{sec:prooftheorem3} we prove Theorem~\ref{thmA} assuming those key results. Finally, in \S\ref{sec:nodal} we prove the convergence of nodal points, in \S\ref{sec:convcar} we prove the convergence of Carathéodory loops and in \S\ref{sec:keyprop} we prove the Key Proposition.

\section{Definitions and key results\label{sec:keyresults}}

\subsection{Dendrites}

A {\em dendrite} $\mj\subset \C$ is a connected and locally connected compact set containing no simple closed curve. 
In this article, we assume in addition that $\mj$ is not reduced to a point. 
Properties of dendrites are discussed in \cite{whyburn}. 
In particular, a dendrite is uniquely arcwise connected.

\begin{definition}
Given two points $\beta$ and $\beta'$ in a dendrite $\mj$, we denote $[\beta,\beta']_{\mj}$ the arc joining $\beta$ and $\beta'$ in $\mj$.
\end{definition}

\begin{definition}
The {\em nodal point} of three points  $\beta$, $\beta'$ and $\beta''$ in a dendrite  $\mj$  is the unique point which belongs simultaneously to $[\beta,\beta']_{\mj}$, $[\beta',\beta'']_{\mj}$ and $[\beta'',\beta]_{\mj}$.
\end{definition}

If $\xi\in \mj$, then any connected component ${\cal C}$ of $\mj\setminus \{\xi\}$ is open in $\mj$ and $\overline {\cal C}\setminus {\cal C} = \{\xi\}$. 

\begin{definition}
A point $\xi$ in a dendrite $\mj$ is {\em non-separating} if $\mj\setminus \{\xi\}$ is connected.
\end{definition}

\begin{definition}
A point $\xi$ in a dendrite $\mj$ is a {\em branching point} if $\mj\setminus \{\xi\}$ has (at least) three connected components. It {\em separates} $\beta$, $\beta'$ and $\beta''$ in $\mj$ if $\beta$, $\beta'$ and $\beta''$ are in three distinct connected components of $\mj\setminus\{\xi\}$.
\end{definition}

Note that a nodal point is not necessarily a branching point. For example, when $\beta''\in [\beta,\beta']_\mj$, then the nodal point of $\beta$, $\beta'$ and $\beta''$ in $\mj$ is the point $\beta''$. In fact, the nodal point $\xi$ of $\beta$, $\beta'$ and $\beta''$ in $\mj$ separates $\beta$, $\beta'$ and $\beta''$ in $\mj$ if and only if $\xi\notin\{\beta, \beta', \beta''\}$.

The complement $\C\setminus \mj$ of a dendrite is connected and simply connected.  It follows that there is a (unique) conformal representation $\phi:\C\setminus \overline \D\to \C\setminus \mj$ such that $\phi(z) / z \to R>0$ as $z\to \infty$. Since $\mj$ is locally connected, a theorem by Torhorst \cite{torhorst}, based on the work on prime-ends of Carathédory (see the discussion in \cite{lasse}), asserts that $\phi$ extends  to a continuous map $\phi:\C\setminus \D\to \C$. 
Since $\mj$ has empty interior, this map is surjective and restricts to a map $\varphi:\s^1\to \mj$. 

\begin{definition}
The {\em Carathéodory loop} of $\mj$ is the restriction $\varphi:\s^1\to \mj$.
\end{definition}

The Carathéodory loop is a continuous and surjective map from $\s^1$ to $\mj$. 
We shall use the following result whose proof is given in \S\ref{sec:nodal}.

\begin{lemma}[Convergence of nodal points]\label{lem:convnodal}
Let $\{\mj_n\}$ be a sequence of dendrites. Assume the associated sequence $\{\varphi_n:\s^1\to \mj_n\}$ of Carathéodory loops converges uniformly to some non constant map $\varphi:\s^1\to \C$. Then, 
\begin{itemize}
\item $\mj:=\varphi(\s^1)$ is a dendrite with  Carathéodory loop $\varphi$; 
\item if $\{\beta_n\}$, $\{\beta'_n\}$ and $\{\beta''_n\}$ are sequences of points in $\mj_n$ converging to $\beta$, $\beta'$ and $\beta''$ in $\mj$, then the corresponding sequence of nodal points of $\beta_n$, $\beta'_n$ and $\beta''_n$ in $\mj_n$ converges to the nodal point of $\beta$, $\beta'$ and $\beta''$ in $\mj$. 
\end{itemize}
\end{lemma}

\subsection{Admissible polynomials}

Consider the affine space $\PP$ of monic cubic polynomials fixing $0$. 
We will restrict our study to the open subset $\VV\subset \PP$ of polynomials $f$ such that:
\begin{enumerate}

\item \label{item:V1} $f$ has three distinct repelling fixed point $\alpha=0$, $\beta$ and $\gamma$; 

\item \label{item:V2} $f$ has two distinct critical points $\omega_\alpha$, $\omega_\beta$; 

\item \label{item:V3} the ray of angle $0$ does not bifurcate and lands at $\alpha$; 
\item \label{item:V4} the ray of angle $1/2$ does not bifurcate and lands at $\beta$; 
\item \label{item:V5} the rays of angles $1/4$ and $-1/4$ do not bifurcate, land at $\gamma$, and separate the plane in two connected components
\begin{itemize}
\item $U_\alpha$ containing  $\alpha$, $\omega_\alpha$ and $f(\omega_\beta)$ and 
\item $U_\beta$ containing $\beta$, $\omega_\beta$, $f(\omega_\alpha)$ and $f^{\circ 2}(\omega_\alpha)$. 
\end{itemize}
\end{enumerate}

Figure \ref{fig:VV} shows the ray configuration and the position of the points involved in the definition of $\VV$. 

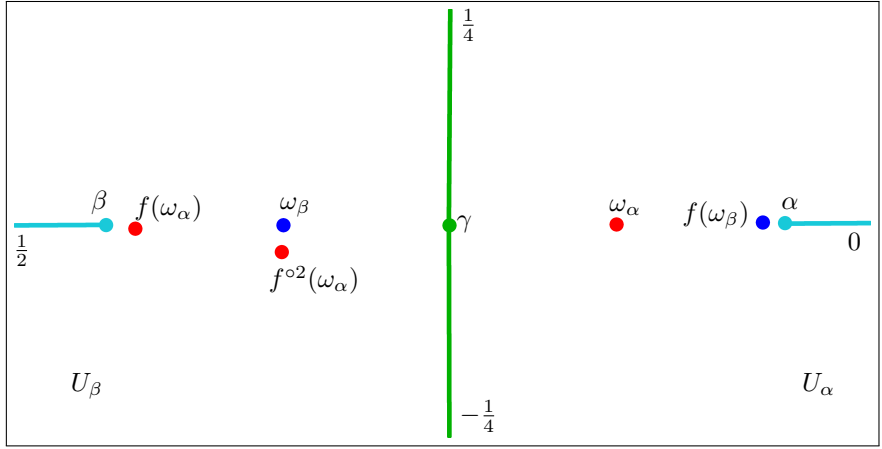
\begin{figure}[htb]
\setlength{\unitlength}{0.002cm}
\centerline{
\begin{picture}(5688,2844)(0,0)
\put(0,0){\fbox{\includegraphics[width=11.376cm]{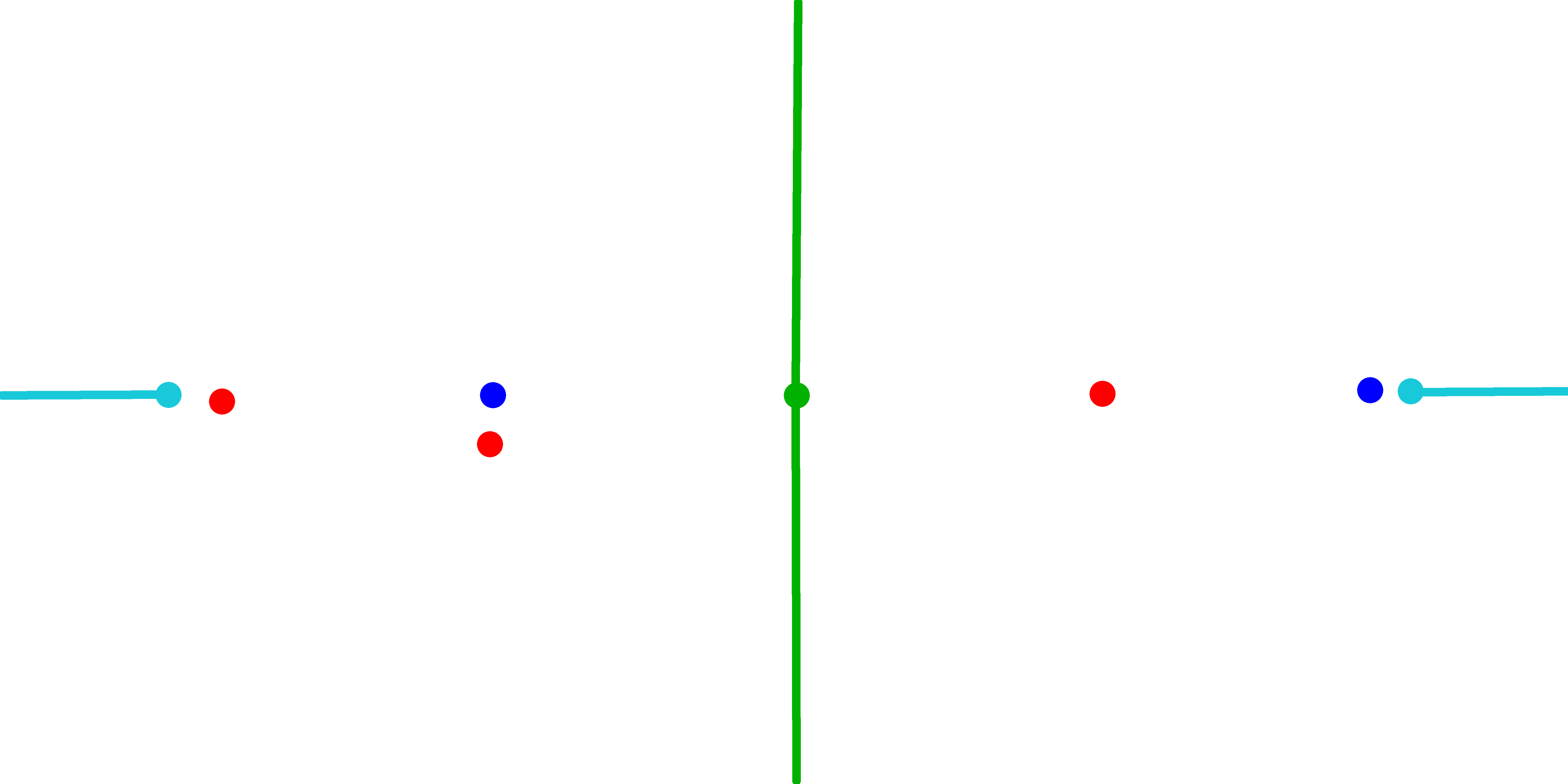}}}
\put(432,312){$U_\beta$}
\put(5296,312){$U_\alpha$}
\put(4016,1504){$\omega_\alpha$}
\put(1824,1520){$\omega_\beta$}
\put(5160,1520){$\alpha$}
\put(568,1520){$\beta$}
\put(3000,1400){$\gamma$}
\put(856,1480){$f(\omega_\alpha)$}
\put(1750,1000){$f^{\circ 2}(\omega_\alpha)$}
\put(4500,1440){$f(\omega_\beta)$}
\put(3020,2700){$\frac{1}{4}$}
\put(3020,100){$-\frac{1}{4}$}
\put(50,1200){$\frac{1}{2}$}
\put(5600,1250){$0$}
\end{picture}}
\caption{  The ray configuration and the position of the points involved in the definition of $\VV$. 
The angle of each external ray is indicated. 
\label{fig:VV}}
\end{figure}

\begin{definition}\label{def:admissible}
A cubic polynomial $f\in \VV$ is {\em admissible} if it has critical points $\omega$ and $\omega'$ and a branching point $\xi$, such that:
\begin{enumerate}
\item\label{item:admis1} $\xi$ is precritical to $\omega$, $\omega$ is precritical to $\omega'$ and  $\omega'$ is prefixed to $\alpha$;
\item\label{item:admis2}  $\xi$ separates $\beta$, $\beta'$ and $\beta''$ in $\mj_f$, where $f^{-1}(\beta) =\{\beta,\beta',\beta''\}$.
\end{enumerate}
We set $\xi_f:=\xi$ and denote by $j_f$ the integer $j\geq 0$ such that $f^{\circ j}(\xi) = \omega$. 
\end{definition}



\begin{figure}[htb]
\setlength{\unitlength}{1cm}
\centerline{
\begin{picture}(11.64,2.22)(0,0)
\put(0,0){\includegraphics[width=11.64cm]{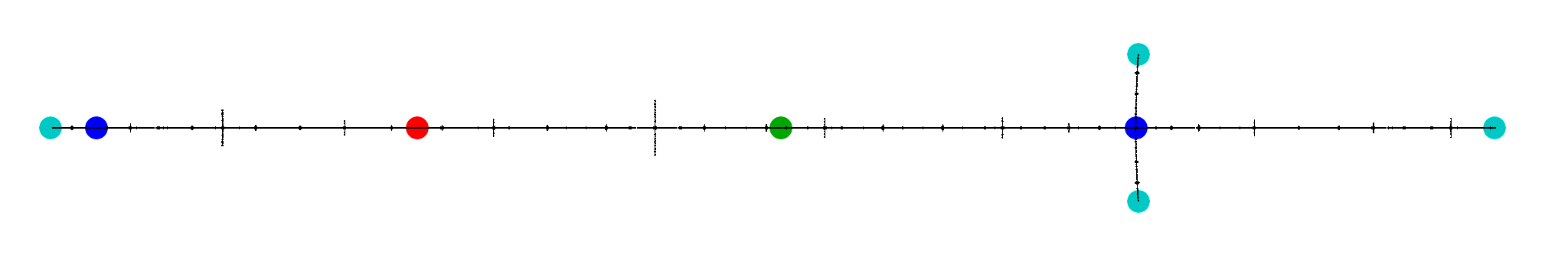}}
\put(2.5,1.10){${\omega'=f^{\circ 2}(\omega)}$}
\put(0.6,1.10){${f(\omega)}$}
\put(8.7,1.10){${\omega=\xi}$}
\put(0.15,1.10){${\beta}$}
\put(8.53,0.10){${\beta''}$}
\put(8.53,1.65){${\beta'}$}
\put(10.5,1.10){${\alpha=f(\omega')}$}
\put(5.7,1.15){$\gamma$}
\end{picture}}
\caption{The Julia set of an admissible polynomial $f$ with critical points $\omega$ and $\omega'$. We have $f^{\circ 2}(\omega) =  \omega'$ and $f(\omega')=\alpha$. In this example, $j_f=0$. \label{fig:suitable}}
\end{figure}
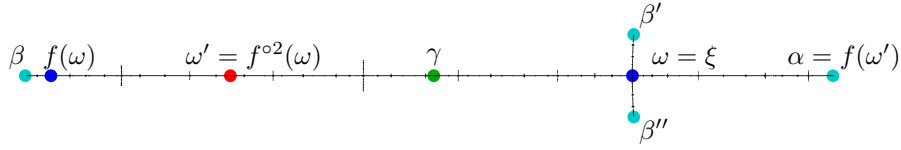

Our key result, proved in \S\ref{sec:keyprop}, is the following. 

\begin{proposition}[Key proposition]\label{prop:key}
Assume $f\in \VV$ is an admissible polynomial. Then, there is a sequence  $\{g_m\}$ of admissible polynomials which converges to $f$, such that $\{\xi_{g_m}\}$ converges to $\xi_f$ and  $j_{g_m}>j_f$ for all $m$. 
\end{proposition}

This result shall be completed with the following observation. 
An admissible polynomial is strictly postcritically finite. It follows that its Julia set is a dendrite (see Figure~\ref{fig:suitable}) and thus, has an associated Carathéodory loop. We shall prove in \S\ref{sec:convcar} that the convergence of polynomials implies the convergence of the associated Carathéodory loops. 

\begin{lemma}[Convergence of Carathéodory loops]\label{lem:convcaratheodory}
Let $\{g_m\}$ be a sequence of cubic polynomials, with locally connected Julia sets, which converges to an admissible polynomial $f$. Then, the sequence of Carathéodory loops of $\mj_{g_m}$ converges uniformly to the Carathéodory loop of $\mj_f$.
\end{lemma}

\section{Existence of non preperiodic and non precritical branching points\label{sec:prooftheorem3}}

We now prove Theorem \ref{thmA}, assuming Lemmas \ref{lem:convnodal} and \ref{lem:convcaratheodory} and Proposition \ref{prop:key}. 

\subsection{An admissible polynomial}\label{sec:admisspol}

We first prove that there exists an admissible polynomial.   
The polynomial $f:\C\to \C$ defined by 
\[f(z) = z(z-3\omega)^2\]
is monic, fixes $0$, has critical points at $\omega$ and $\omega'=3\omega$ and satisfies $f(\omega')=0$. 
Then, 
\[f(\omega) = 4\omega^3\quad\text{and}\quad f^{\circ 2}(\omega) = 4\omega^3(4\omega^3-3\omega)^2.\]
The equation $f^{\circ 2}(\omega)=3\omega$ has a unique real negative solution which is 
\[\omega:= -\frac{1}{4}\sqrt{6+2\sqrt{9+8\sqrt{3}}}.\]
The graph of $f:\R\to \R$ is shown on Figure \ref{fig:admissible} and the Julia set $\mj_f$ is shown in Figure \ref{fig:suitable}.

The fixed points of $f$ are $\alpha:=0$, $\beta:=3\omega-1$ and $\gamma:=3\omega+1\in \left(\beta, \alpha\right)$. 
The respective multipliers at $\alpha$, $\beta$ and $\gamma$ are $9\omega^2>1$, $3-6\omega>1$ and $3+6\omega<-1$. In particular, the three fixed points are repelling.


\definecolor{ccqqqq}{rgb}{0.8,0.,0.}
\definecolor{uuuuuu}{rgb}{0.26666666666666666,0.26666666666666666,0.26666666666666666}
\definecolor{qqwuqq}{rgb}{0.,0.39215686274509803,0.}

\begin{figure}[htb]
\centerline{
\begin{tikzpicture}[line cap=round,line join=round,>=triangle 45,x=1.5cm,y=1.5cm]
\clip(-5.098308318307095,-4.084167793755014) rectangle (0.8989823231408591,0.4875373673487669);
\draw[line width=1.2pt,color=qqwuqq,smooth,samples=100,domain=-5.098308318307095:0.8989823231408591] plot(\x,{(\x)*((\x)-3.0*(-0.9862072184965908))^(2.0)});
\draw[line width=1.2pt,color=ccqqqq,smooth,samples=100,domain=-5.098308318307095:0.8989823231408591] plot(\x,{(\x)});
\draw (-3.958621655489772,0.)-- (-3.958621655489772,-3.958621655489772);
\draw (-3.958621655489772,-3.958621655489772)-- (0.,-3.958621655489772);
\draw (0.,-3.958621655489772)-- (0.,0.);
\draw (0.,0.)-- (-3.958621655489772,0.);
\draw (-2.958621655489772,0.)-- (-2.9586216554897815,-2.9586216554897815);
\draw (-0.9862072184965908,0.)-- (-0.9862072184965908,-3.836759016017956);
\draw (-1.958621655,0.)-- (-1.958621655,-1.958621655);
\draw (-0.9862072184965908,-3.836759016017956)-- (-3.836759016017956,-3.836759016017956);
\draw (-3.836759016017956,-3.836759016017956)-- (-3.836759016017958,-2.958621655489782);
\draw (-3.836759016017958,-2.958621655489782)-- (-2.9586216554897815,-2.9586216554897815);
\draw (0.06329428293909502,0.16800958727162096) node[anchor=north west] {$\alpha=0$};
\draw (-1.1001145573417923,0.3) node[anchor=north west] {$\omega$};
\draw (-3.0418602978105973,0.36) node[anchor=north west] {$\omega'$};
\draw (-4,0.33) node[anchor=north west] {$f(\omega)$};
\draw (-4.213462158093463,0.2335537472874458) node[anchor=north west] {$\beta$};
\draw (-2.1,0.3) node[anchor=north west] {$\gamma$};
\begin{scriptsize}
\draw [fill=uuuuuu] (-3.958621655489772,0.) circle (1.5pt);
\draw [fill=uuuuuu] (0.,0.) circle (1.5pt);
\draw [fill=uuuuuu] (-2.958621655489772,0.) circle (1.5pt);
\draw [fill=uuuuuu] (-1.958621655,0.) circle (1.5pt);
\draw [fill=uuuuuu] (-0.9862072184965908,0.) circle (1.5pt);
\draw [fill=uuuuuu] (-3.836759016017957,0.) circle (1.5pt);
\end{scriptsize}
\end{tikzpicture}
}
\caption{The graph of the real polynomial $f(z) = z(z-3\omega)^2$.
\label{fig:admissible}}
\end{figure}
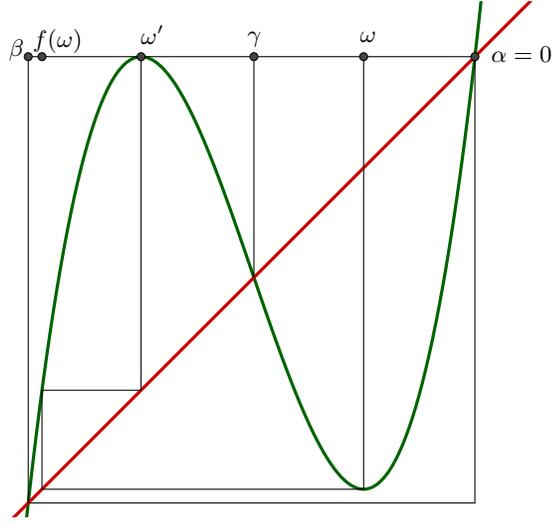

We have that 
\[\beta=3\omega-1 < f(\omega) = 4\omega^3< f^{\circ 2}(\omega) = \omega' = 3\omega<\gamma = 3\omega+1< \omega < 0=\alpha.\]
The intersection of the Julia $\mj_f$ with the real axis is the interval $[\beta,\alpha]$. Since $f$ is a real polynomial, the ray of angle 0 is $(\alpha,+\infty)$ and lands at $\alpha$. The ray of angle $1/2$ is $(-\infty,\beta)$ and lands at $\beta$. Since the multiplier of $\gamma$ is negative, the orbit of a ray landing at $\gamma$ must alternate in between the upper half-plane and the lower half-plane. Since $f(\omega') = \alpha$, the rays of angles $1/3$ and $-1/3$ land at $\omega'$ and separate $\beta$ from $\gamma$. It follows that the digits of  ternary expansion of the rays landing at $\gamma$ alternate between $0$ and $2$: those rays have angle $1/4$ and $-1/4$ and separate the plane in two connected components. The one containing $\alpha$ contains $\omega$ and $f(\omega')$. The one containing $\beta$ contains $\omega'= f^{\circ 2}(\omega)$ and  $f(\omega)$. So,  the polynomial $f$ belongs to $\VV$.

Set $\xi:= \omega$. Then, $\xi$ is precritical to $\omega$, $\omega$ is precritical to $\omega'$ and $\omega'$ is prefixed to $\alpha:=0$. So, Condition \eqref{item:admis1}   in Definition \ref{def:admissible} is satisfied with $j_f=0$. 
In addition, $\xi$ separates $\beta$, $\beta'$ and $\beta''$ in $\mj_f$ (see Figure \ref{fig:omegaseparates}), so that 
Condition \eqref{item:admis2}   in Definition \ref{def:admissible} is satisfied. 

\definecolor{xdxdff}{rgb}{0.49019607843137253,0.49019607843137253,1.}
\definecolor{ffqqqq}{rgb}{1.,0.,0.}
\definecolor{qqqqff}{rgb}{0.,0.,1.}

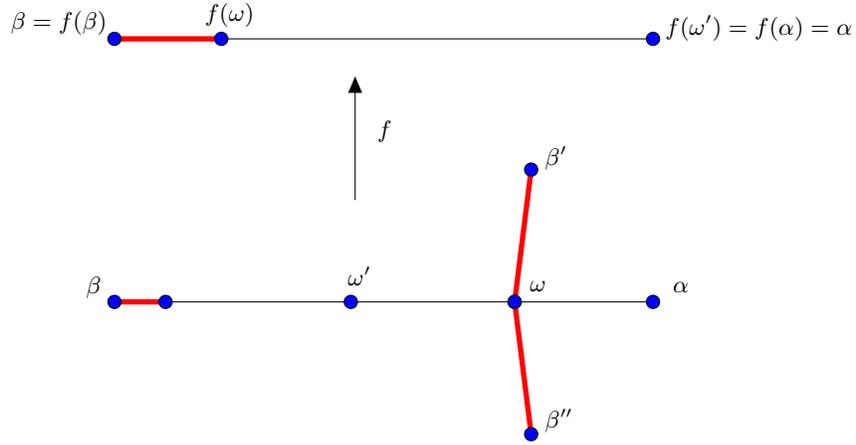
\begin{figure}[htbp]
\centerline{
\begin{tikzpicture}[line cap=round,line join=round,>=triangle 45,x=1.0cm,y=1.0cm]
\draw (7.76,3.5)-- (2.02,3.5);
\draw [line width=2.pt,color=ffqqqq] (2.02,3.5)-- (0.6,3.5);
\draw (7.76,0.)-- (5.92,0.);
\draw [line width=2.pt,color=ffqqqq] (6.14,1.76)-- (5.92,0.);
\draw (5.92,0.)-- (1.28,0.);
\draw [line width=2.pt,color=ffqqqq] (1.28,0.)-- (0.6,0.);
\draw [line width=2.pt,color=ffqqqq] (6.14,-1.76)-- (5.92,0.);
\draw (7.8,3.9350957558601176) node[anchor=north west] {$f(\omega')=f(\alpha)=\alpha$};
\draw (7.9,0.3886265657764875) node[anchor=north west] {$\alpha$};
\draw (-0.9,3.9932345950418164) node[anchor=north west] {$\beta = f(\beta)$};
\draw (0.1,0.46130011475361105) node[anchor=north west] {$\beta$};
\draw [->] (3.8,1.36) -- (3.8,3.);
\draw (3.973382702479108,2.539763615499345) node[anchor=north west] {$f$};
\draw (1.6768985548019981,4.1) node[anchor=north west] {$f(\omega)$};
\draw (6.2,2.2) node[anchor=north west] {$\beta'$};
\draw (6,0.4) node[anchor=north west] {$\omega$};
\draw (3.58094553800264,0.6) node[anchor=north west] {$\omega'$};
\draw (6.2,-1.3) node[anchor=north west] {$\beta''$};
\begin{scriptsize}
\draw [fill=qqqqff] (7.76,3.5) circle (2.5pt);
\draw [fill=qqqqff] (2.02,3.5) circle (2.5pt);
\draw [fill=qqqqff] (0.6,3.5) circle (2.5pt);
\draw [fill=qqqqff] (7.76,0.) circle (2.5pt);
\draw [fill=qqqqff] (5.92,0.) circle (2.5pt);
\draw [fill=qqqqff] (6.14,1.76) circle (2.5pt);
\draw [fill=qqqqff] (1.28,0.) circle (2.5pt);
\draw [fill=qqqqff] (0.6,0.) circle (2.5pt);
\draw [fill=qqqqff] (3.742569832402235,0.) circle (2.5pt);
\draw [fill=qqqqff] (6.14,-1.76) circle (2.5pt);
\draw [fill=qqqqff] (5.92,0.) circle (2.5pt);
\end{scriptsize}
\end{tikzpicture}
}
\caption{A schematic representation of the preimage of $[\beta,\alpha]$ by $f$ illustrating why $\omega$ separates $\beta$, $\beta'$ and $\beta''$ in $\mj_f$\label{fig:omegaseparates}}
\end{figure}

\subsection{A sequence of admissible polynomials}

We now build a Cauchy sequence of admissible polynomials. Given two maps $f:X\to \C$ and $g:X\to \C$ defined on some set $X\subseteq \C$, we denote by $d_X(f,g)$ the spherical distance between $f$ and $g$:  
\[d_X(f,g):= \sup_{x\in X} \frac{\bigl|f(x)-g(x)|}{\sqrt{1+|f(x)|^2}\sqrt{1+|g(x)|^2}}.\]

Let $f$ be an admissible polynomial and let $\xi_f$ be the associated branching point. 
According to Definition~\ref{def:admissible}, the points 
\[\xi_f,f(\xi_f),\ldots,f^{\circ (j_f-1)}(\xi_f)\]
are not critical points of $f$ and are pairwise distinct. For each integer $j$, the map $\mathfrak{V}\times \C\rightarrow \C: (g,z)\mapsto g^{\circ j}(z)$ is continuous at $(f,\xi_f)$ and the critical points of $g$ depend continuously on $g$. Thus, there exists $\eps_f>0$ such that for $|z-\xi_f|<\eps_f$ and $d_\C(f,g)<\eps_f$, the points 
\[z,g(z),\ldots,g^{\circ (j_f-1)}(z)\]
are not critical points of $f$ and are pairwise distinct. 

According to Proposition \ref{prop:key}, there is a sequence $\{g_m\}$ of admissible polynomials which converges to $f$, such that $\{\xi_{g_m}\}$ converges to $\xi_f$ and  $j_{g_m}>j_f$ for all $m$. According to Lemma \ref{lem:convcaratheodory}, the sequence of Carathéodory loops $\varphi_{g_m}$ of $\mj_{g_m}$ converges uniformly to the Carathéodory loop $\varphi_f$ of $\mj_f$. So, for all $\vareps>0$, if $m$ is large enough, then 
\[|\xi_{g_m}-\xi_f|<\vareps,\quad d_\C(f,g_m)<\vareps\quad\text{and}\quad d_{\s^1}(\varphi_f,\varphi_{g_m})<\vareps.\]

Let us now define recursively a sequence $\{f_n\}$ of admissible polynomials  and a sequence $\{\vareps_n\}$ of positive numbers  as follows. We let $f_0$ be any admissible polynomial and set $\vareps_0:=\eps_{f_0}/2$. Once $f_n$ and $\vareps_n$ are defined, we let $f_{n+1}$ be an admissible polynomial such that $j_{f_{n+1}}>j_{f_n}$, 
\[|\xi_{f_{n+1}}-\xi_{f_n}|<\vareps_n,\quad d_\C(f_n,f_{n+1})<\vareps_n\quad\text{and}\quad   d_{\s^1}(\varphi_{f_n},\varphi_{f_{n+1}})<\vareps_n.\]
We then set 
\[\vareps_{n+1} = \frac{1}{2} \min \bigl( \vareps_n,\eps_{f_{n+1}}\bigr).\]

Figures \ref{fig:sequenceperturbations} and \ref{fig:sequenceperturbationszoom} illustrate the Julia sets of $f_0$, $f_1$, $f_2$, $f_3$ for such a sequence.   Observe that for all $n\geq 0$ and all $p\geq 0$, $\vareps_{n+p}\leq \vareps_n/2^p\leq \eps_{f_n}/2^{p+1}$. As a consequence, 
\[|\xi_{f_{n+p}}-\xi_{f_n}|<2\vareps_n\leq  \eps_{f_n},\quad d_\C(f_n,f_{n+p})<2\vareps_n\leq\eps_{f_n}\]
and
\[d_{\s^1}(\varphi_{f_n},\varphi_{f_{n+p}})<2\vareps_n\leq\eps_{f_n}.\]
In particular, the sequences $\{f_n\}$,  $\{\xi_{f_n}\}$, and $\{\varphi_{f_n}\}$ are Cauchy sequences.

\subsection{Proof of Theorem~\ref{thmA}} 

The sequences $\{f_n\}$,  $\{\xi_{f_n}\}$, and $\{\varphi_{f_n}\}$ converge to $f$, $\xi$ and $\varphi$. We now prove that $\xi$ is a wandering non-precritical branching point of $f$.

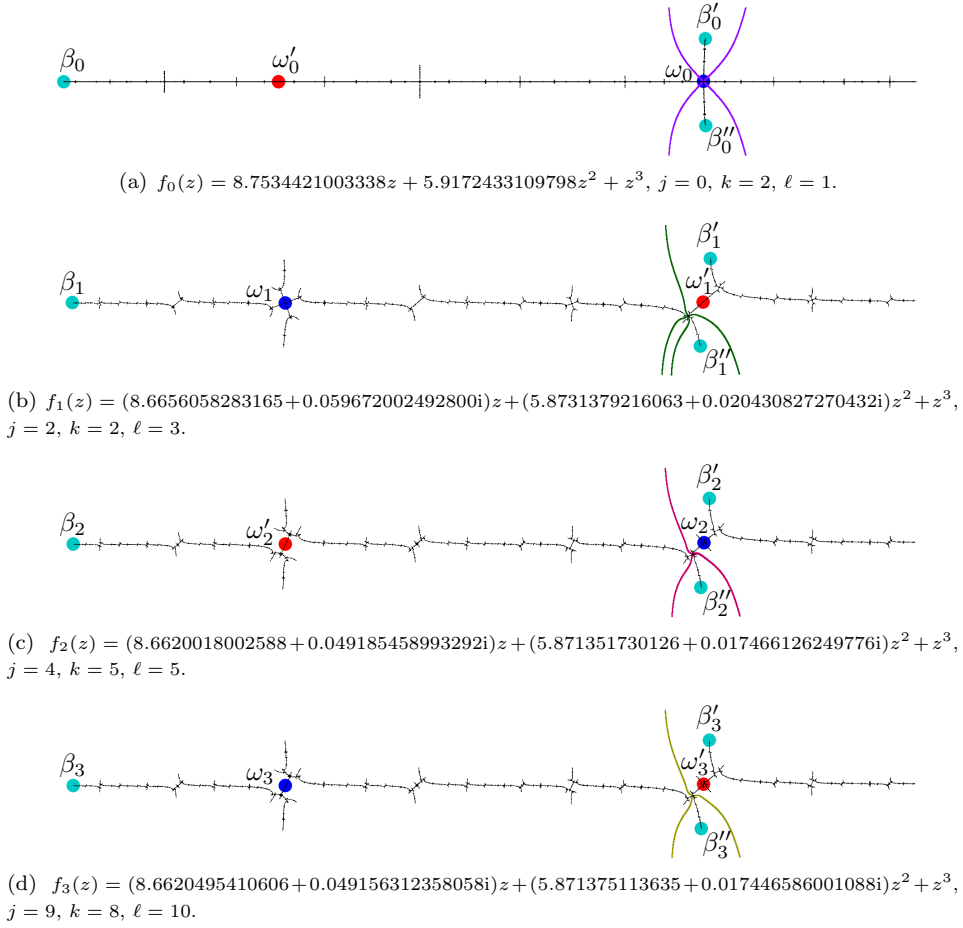
\begin{figure}[htbp]
    \centering
 
       \subfigure[\scriptsize{$f_0(z)=8.7534421003338z+5.9172433109798z^2+z^3$, $j=0,$ $k=2$, $\ell=1$.}  ]{
	\def\svgwidth{350pt}
	 \setlength{\unitlength}{\svgwidth}
	    \begin{picture}(1,0.16285714)%
    \put(0,0){\includegraphics[width=\unitlength,page=1]{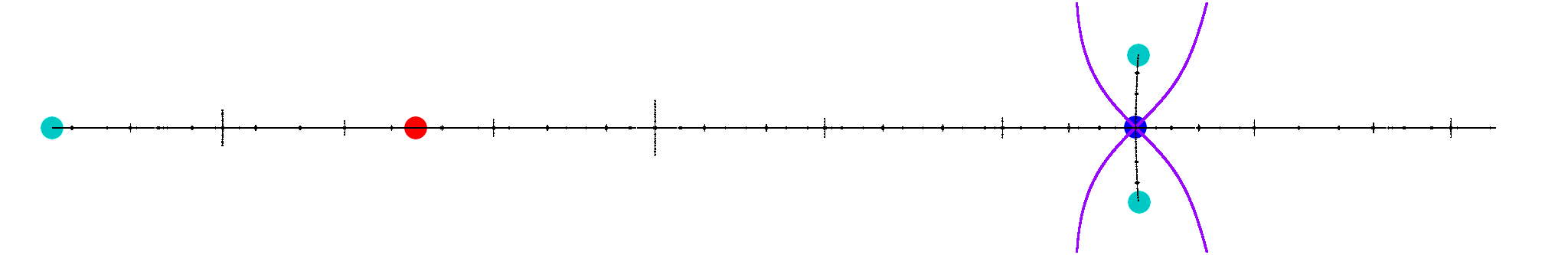}}%
    \put(0.02806122,0.09714286){\color[rgb]{0,0,0}\makebox(0,0)[lb]{\smash{$\beta_0$}}}%
     \put(0.25806122,0.09714286){\color[rgb]{0,0,0}\makebox(0,0)[lb]{\smash{$\omega'_0$}}}%
      \put(0.68306122,0.08614286){\color[rgb]{0,0,0}\makebox(0,0)[lb]{\smash{$\omega_0$}}}%
    \put(0.71598637,0.14489796){\color[rgb]{0,0,0}\makebox(0,0)[lb]{\smash{$\beta'_0$}}}%
    \put(0.72713375,0.01109456){\color[rgb]{0,0,0}\makebox(0,0)[lb]{\smash{$\beta''_0$}}}%
  \end{picture}%
   }
    \hspace{0.1in}
  
    \subfigure[\scriptsize{$f_1(z)=(8.6656058283165+0.059672002492800{\rm i})z+(5.8731379216063+0.020430827270432{\rm i})z^2+z^3$, $j=2,$ $k=2$, $\ell=3$.}  ]{
   	\def\svgwidth{350pt}
	 \setlength{\unitlength}{\svgwidth}
	    \begin{picture}(1,0.16285714)%
    \put(0,0){\includegraphics[width=\unitlength,page=1]{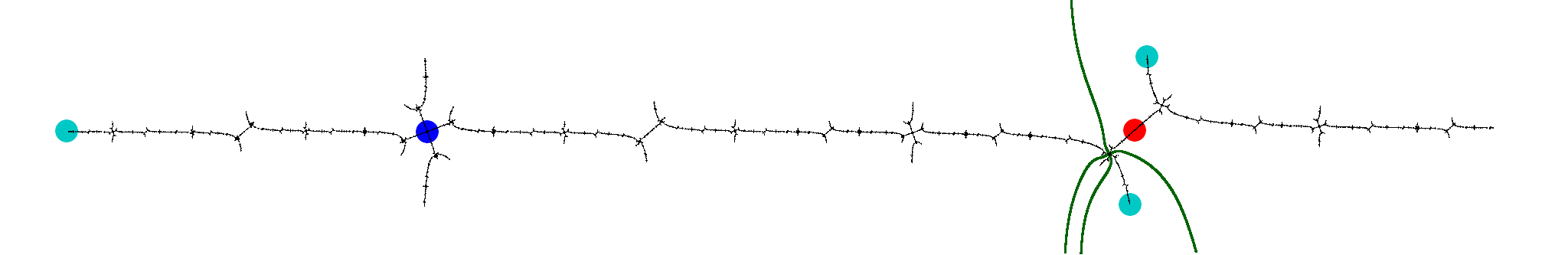}}%
    \put(0.02806122,0.093){\color[rgb]{0,0,0}\makebox(0,0)[lb]{\smash{$\beta_1$}}}%
     \put(0.23,0.085){\color[rgb]{0,0,0}\makebox(0,0)[lb]{\smash{$\omega_1$}}}%
      \put(0.705,0.095){\color[rgb]{0,0,0}\makebox(0,0)[lb]{\smash{$\omega'_1$}}}%
    \put(0.71598637,0.14489796){\color[rgb]{0,0,0}\makebox(0,0)[lb]{\smash{$\beta'_1$}}}%
    \put(0.723,0.008){\color[rgb]{0,0,0}\makebox(0,0)[lb]{\smash{$\beta''_1$}}}%
  \end{picture}%
   }
    \hspace{0.1in}
     \subfigure[\scriptsize{ $f_2(z)=(8.6620018002588+0.049185458993292{\rm i})z+(5.871351730126+0.017466126249776{\rm i})z^2+z^3$, $j=4,$ $k=5$, $\ell=5$.}  ]{
     	\def\svgwidth{350pt}
	 \setlength{\unitlength}{\svgwidth}
	    \begin{picture}(1,0.16285714)%
    \put(0,0){\includegraphics[width=\unitlength,page=1]{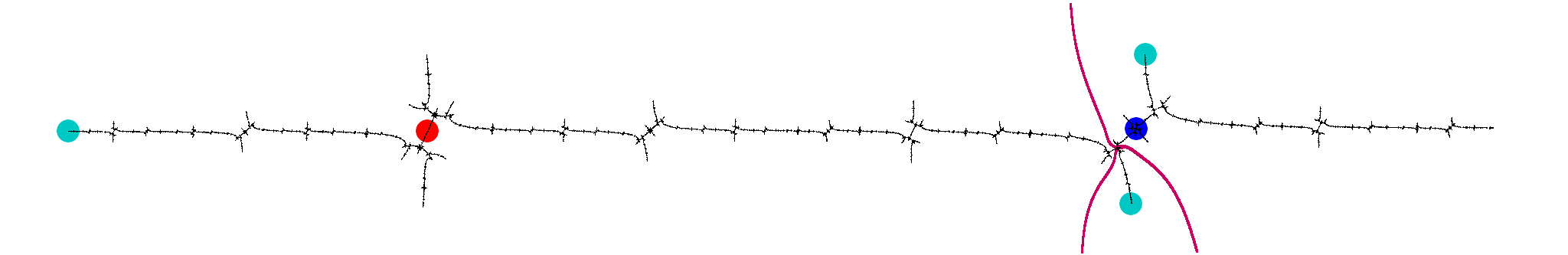}}%
    \put(0.02806122,0.093){\color[rgb]{0,0,0}\makebox(0,0)[lb]{\smash{$\beta_2$}}}%
     \put(0.23,0.085){\color[rgb]{0,0,0}\makebox(0,0)[lb]{\smash{$\omega'_2$}}}%
      \put(0.70,0.095){\color[rgb]{0,0,0}\makebox(0,0)[lb]{\smash{$\omega_2$}}}%
    \put(0.71598637,0.14489796){\color[rgb]{0,0,0}\makebox(0,0)[lb]{\smash{$\beta'_2$}}}%
    \put(0.723,0.008){\color[rgb]{0,0,0}\makebox(0,0)[lb]{\smash{$\beta''_2$}}}%
  \end{picture}%
   }
    \hspace{0.1in}
     \subfigure[\scriptsize{ $f_3(z)=(8.6620495410606+0.049156312358058{\rm i})z+(5.871375113635+0.017446586001088{\rm i})z^2+z^3$, $j=9,$ $k=8$, $\ell=10$.}  ]{
       	\def\svgwidth{350pt}
	 \setlength{\unitlength}{\svgwidth}
	    \begin{picture}(1,0.16285714)%
    \put(0,0){\includegraphics[width=\unitlength,page=1]{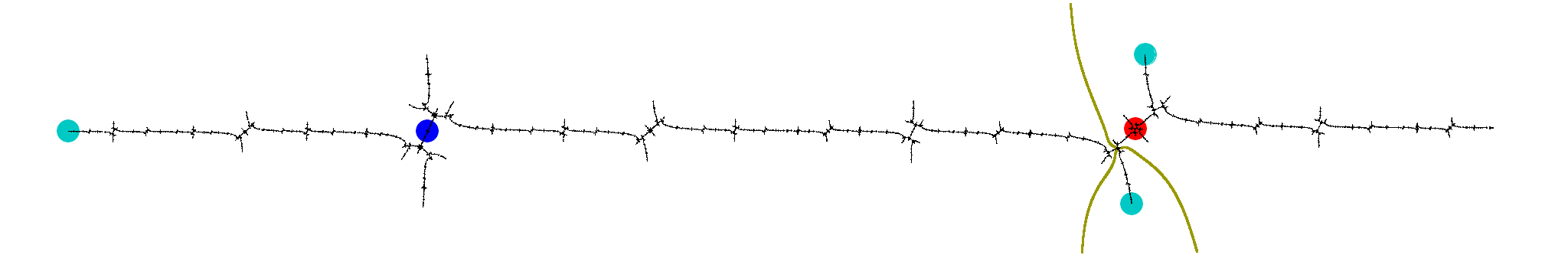}}%
    \put(0.02806122,0.093){\color[rgb]{0,0,0}\makebox(0,0)[lb]{\smash{$\beta_3$}}}%
     \put(0.23,0.085){\color[rgb]{0,0,0}\makebox(0,0)[lb]{\smash{$\omega_3$}}}%
      \put(0.70,0.097){\color[rgb]{0,0,0}\makebox(0,0)[lb]{\smash{$\omega'_3$}}}%
    \put(0.71598637,0.14489796){\color[rgb]{0,0,0}\makebox(0,0)[lb]{\smash{$\beta'_3$}}}%
    \put(0.723,0.008){\color[rgb]{0,0,0}\makebox(0,0)[lb]{\smash{$\beta''_3$}}}%
  \end{picture}%
   }
    \caption{A sequence of 3 perturbations of the polynomial $f_0$ introduced in \S\ref{sec:admisspol}. We draw the external rays landing at the corresponding branching points $\xi$. In all figures $f^{\circ j}(\xi)=\omega$, $f^{\circ k}(\omega)=\omega'$ and $f^{\circ \ell}(\omega')=0$.	\label{fig:sequenceperturbations}}
    \end{figure}

    \begin{figure}[htbp]
    \centering
    \subfigure[\scriptsize{$f_0$}  ]{
    \def\svgwidth{160pt}
	 \setlength{\unitlength}{\svgwidth}
	  \begin{picture}(1,1)%
    \put(0,0){\includegraphics[width=\unitlength,page=1]{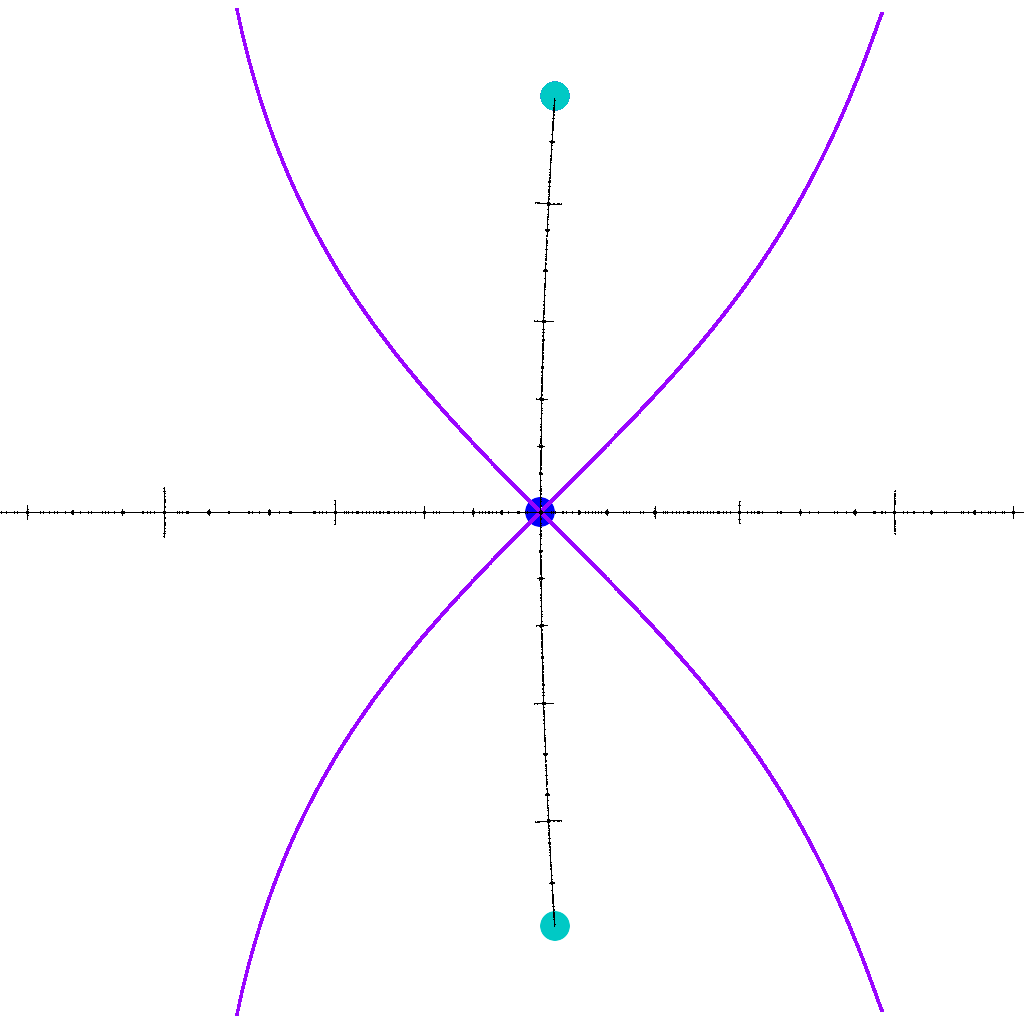}}%
    \put(0.565,0.91525424){\color[rgb]{0,0,0}\makebox(0,0)[lb]{\smash{$\beta'_0$}}}%
        \put(0.44,0.51){\color[rgb]{0,0,0}\makebox(0,0)[lb]{\smash{$\omega_0$}}}%
    \put(0.565,0.07){\color[rgb]{0,0,0}\makebox(0,0)[lb]{\smash{$\beta''_0$}}}%
  \end{picture}%
   }
    \hspace{0.1in}
    \subfigure[\scriptsize{$f_1$}  ]{
    \def\svgwidth{160pt}
	 \setlength{\unitlength}{\svgwidth}
	  \begin{picture}(1,1)%
    \put(0,0){\includegraphics[width=\unitlength,page=1]{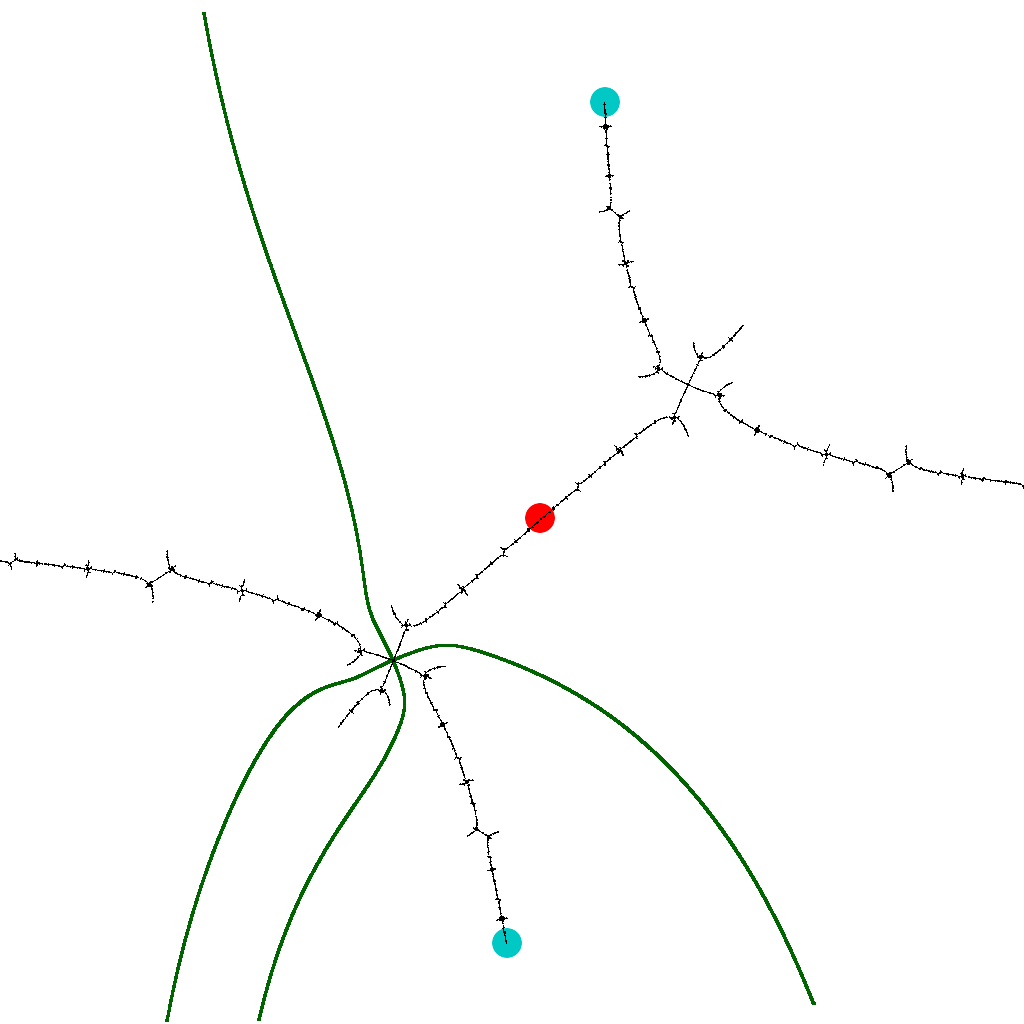}}%
    \put(0.61,0.91525424){\color[rgb]{0,0,0}\makebox(0,0)[lb]{\smash{$\beta'_1$}}}%
        \put(0.44,0.51){\color[rgb]{0,0,0}\makebox(0,0)[lb]{\smash{$\omega'_1$}}}%
    \put(0.515,0.055){\color[rgb]{0,0,0}\makebox(0,0)[lb]{\smash{$\beta''_1$}}}%
  \end{picture}%
   }
    \hspace{0.1in}
     \subfigure[\scriptsize{ $f_2$}  ]{
    \def\svgwidth{160pt}
	 \setlength{\unitlength}{\svgwidth}
	  \begin{picture}(1,1)%
    \put(0,0){\includegraphics[width=\unitlength,page=1]{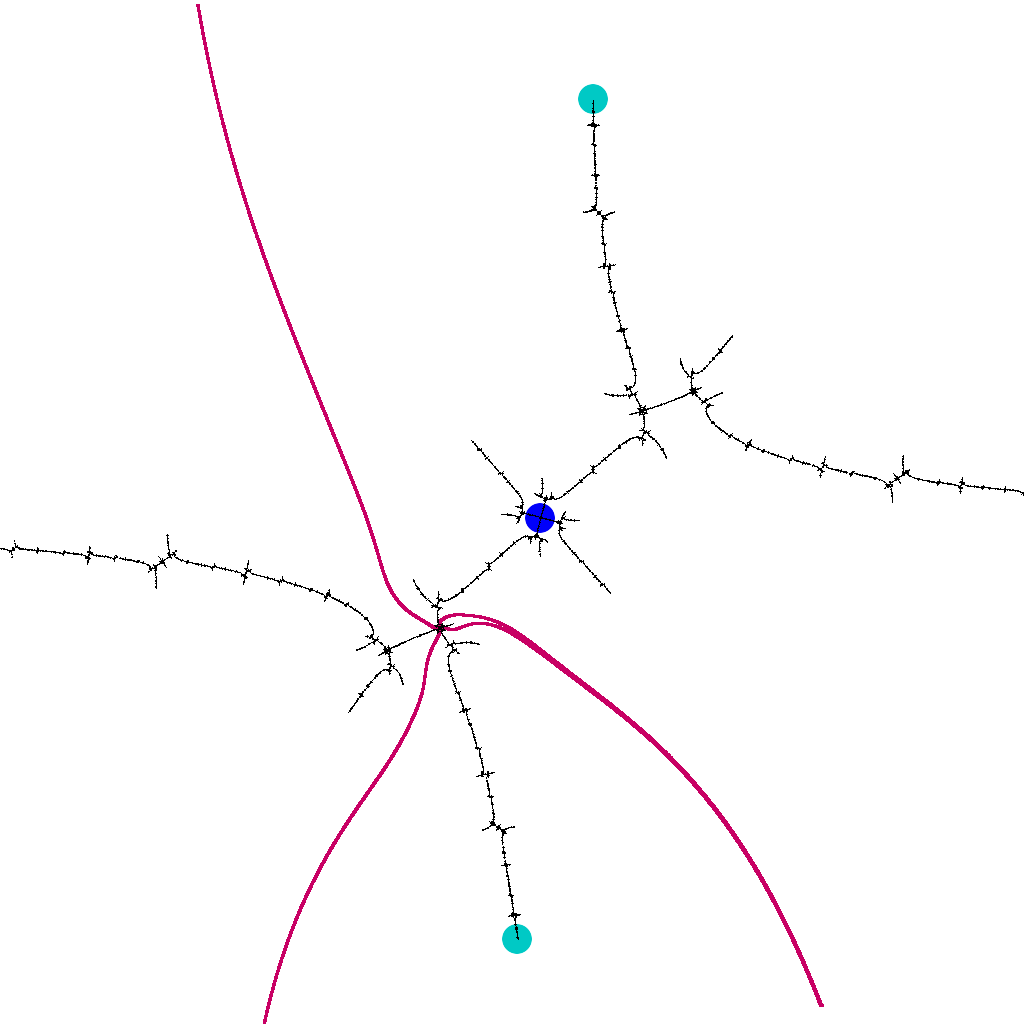}}%
    \put(0.605,0.915){\color[rgb]{0,0,0}\makebox(0,0)[lb]{\smash{$\beta'_2$}}}%
        \put(0.42,0.49){\color[rgb]{0,0,0}\makebox(0,0)[lb]{\smash{$\omega_2$}}}%
    \put(0.525,0.055){\color[rgb]{0,0,0}\makebox(0,0)[lb]{\smash{$\beta''_2$}}}%
  \end{picture}%
   }
    \hspace{0.1in}
     \subfigure[\scriptsize{ $f_3$}  ]{
       \def\svgwidth{160pt}
	 \setlength{\unitlength}{\svgwidth}
	  \begin{picture}(1,1)%
    \put(0,0){\includegraphics[width=\unitlength,page=1]{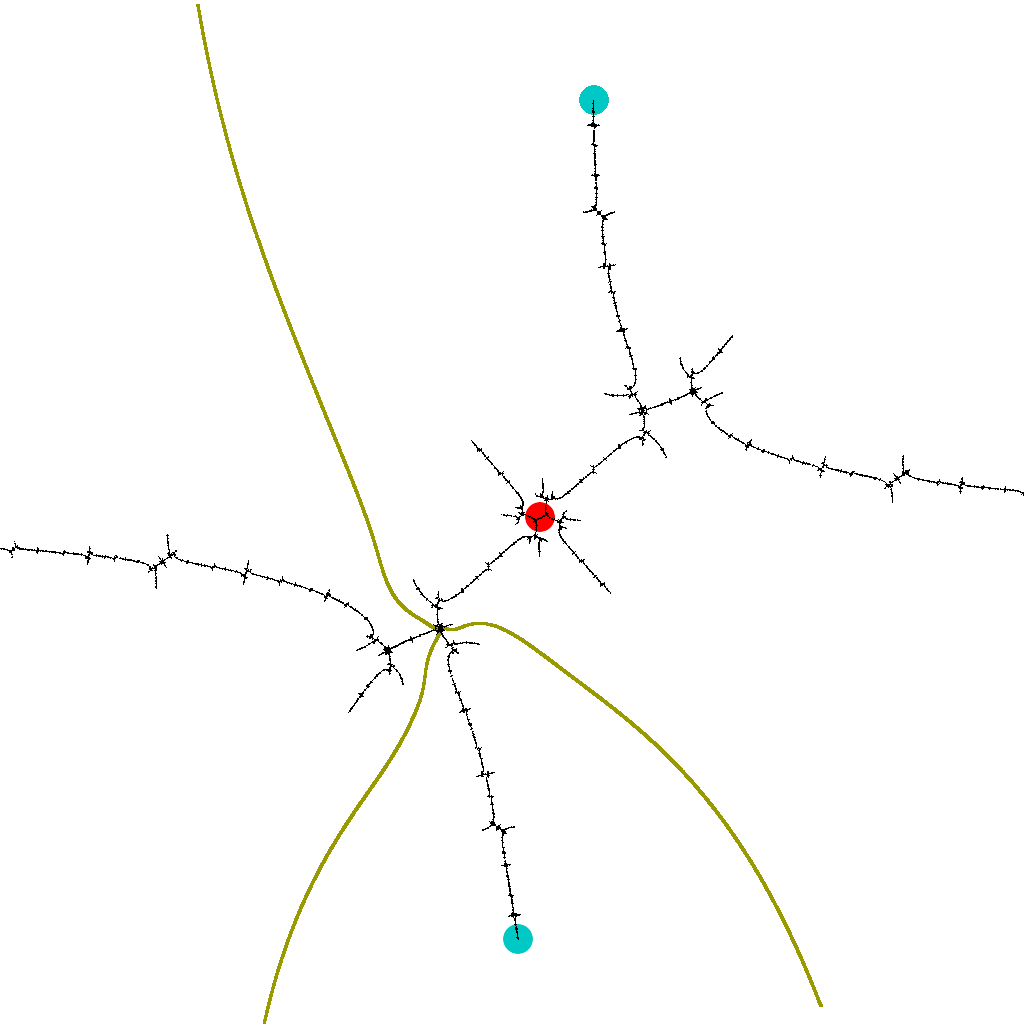}}%
    \put(0.605,0.915){\color[rgb]{0,0,0}\makebox(0,0)[lb]{\smash{$\beta'_3$}}}%
        \put(0.42,0.49){\color[rgb]{0,0,0}\makebox(0,0)[lb]{\smash{$\omega'_3$}}}%
    \put(0.525,0.055){\color[rgb]{0,0,0}\makebox(0,0)[lb]{\smash{$\beta''_3$}}}%
  \end{picture}%
   }
	
    \caption{Zooms in Figure~\ref{fig:sequenceperturbations}.\label{fig:sequenceperturbationszoom} }
    \end{figure}
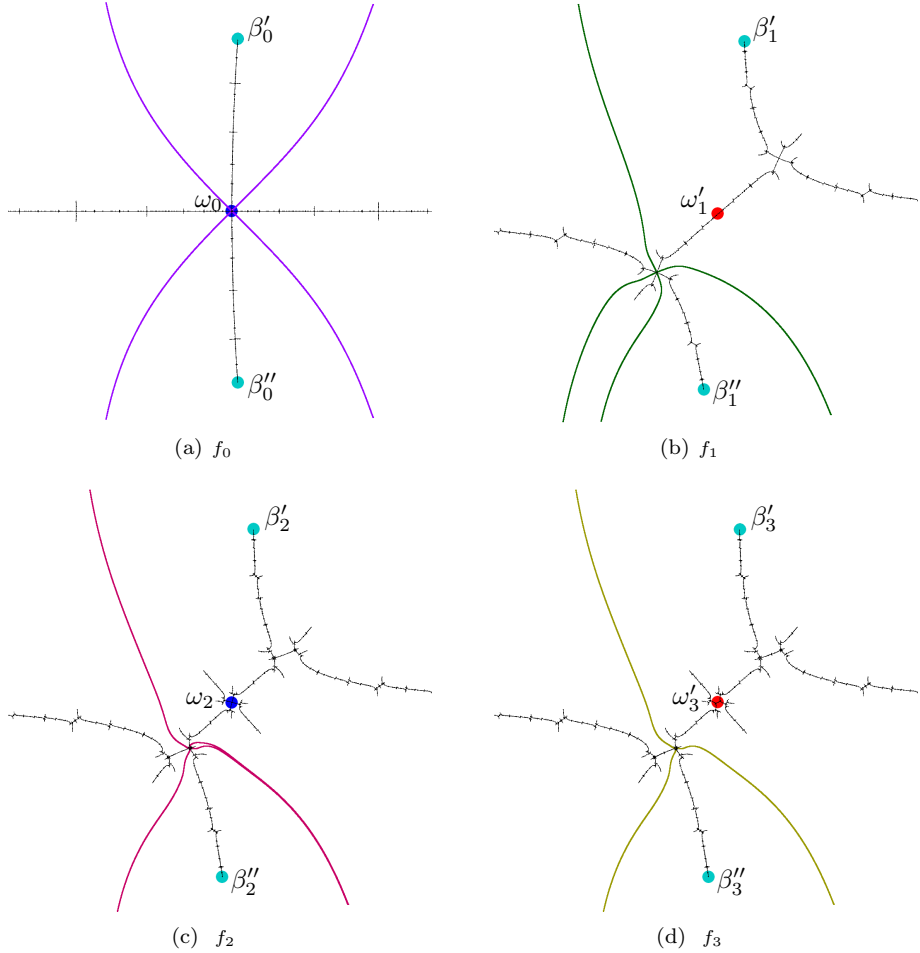

Let $\beta_n$, $\beta'_n$, $\beta''_n$  be the landing points of the rays of angles $1/2$, $1/6$ and $-1/6$ for $f_n$, so that 
\begin{itemize}
\item $f_n^{-1}(\beta_n) = \{\beta_n,\beta'_n,\beta''_n\}$ and 
\item $\xi_n$ is the nodal point of $\beta_n$, $\beta'_n$ and $\beta''_n$ in $\mj_{f_n}$. 
\end{itemize}
The sequences $\{\beta_n\}$, $\{\beta'_n\}$ and $\{\beta''_n\}$ converge to the landing points $\beta$, $\beta'$ and $\beta''$ of the rays of angles $1/2$, $1/6$ and $-1/6$ for $f$. According to Lemma \ref{lem:convnodal}, $\mj:=\varphi(\s^1)$ is a dendrite with Carathéodory loop $\varphi$ and $\xi$ is the nodal point of $\beta$, $\beta'$ and $\beta''$ in $\mj$. 

We claim that $\mj_f=\mj$. Indeed, the critical orbits of $f$ are approximated by the critical orbits of $f_n$. So, they are bounded and $\mj_f$ is connected. According to \cite{Douady2},  
\[\mj_f\subseteq \mj:=\lim_{n\to +\infty} \mj_{f_n}\subseteq \mk_f.\] 
Since $\mj$ is a dendrite and $\mj_f\subseteq \mj$ is connected, $\mj_f$ itself is a dendrite. It follows that $\mk_f = \mj_f = \mj$. 

We claim that $\xi$ is neither preperiodic nor precritical for $f$, so that it cannot coincide with $\beta$, $\beta'$ or $\beta''$. As a consequence, $\xi$ is a branching point which is neither preperiodic nor precritical, as required. 

To prove that $\xi$ is neither preperiodic nor precritical, we use that for all $n$, 
\[|\xi-\xi_{f_n}|< \eps_{f_n}\quad\text{and}\quad d_\C(f,f_n)<\eps_{f_n}.\]
By definition of $\eps_{f_n}$, the points 
\[\xi,f(\xi),\ldots,f^{\circ (j_{f_n}-1)}(\xi)\]
are not critical points of $f$ and are pairwise distinct. Since the sequence $\{j_{f_n}\}$ is increasing, it takes arbitrarily large values. Therefore, all the points in the $f$-orbit of $\xi$ are not critical points of $f$ and are pairwise distinct. 

This completes the proof of Theorem \ref{thmA} assuming Lemmas \ref{lem:convnodal} and \ref{lem:convcaratheodory} and Proposition \ref{prop:key}.

\section{Convergence of nodal points\label{sec:nodal}}

Here, we prove Lemma \ref{lem:convnodal}. 
Let $\{\mj_n\}$ be a sequence of dendrites. Assume the associated sequence $\{\varphi_n:\s^1\to \mj_n\}$ of Carathéodory loops converges uniformly to some non constant map $\varphi:\s^1\to \C$. 

We first prove that $\mj:=\varphi(\s^1)$ is a dendrite. Since $\varphi$ is continuous, $\mj$ is connected and locally connected. 
The Carathéodory loop $\varphi_n:\s^1\to \mj_n$ is the boundary value of a  continuous map $\phi_n:\C\setminus \D\to \C$ which is univalent in $\C\setminus \overline \D$. The sequence  $\{\varphi_n:\s^1\to \mj_n\}$ converges uniformly to some $\varphi:\s^1\to \C$; according to the Maximum Modulus Principle, the sequence 
$\{\phi_n:\C\setminus \D \to \C\}$ converges uniformly to some $\phi:\C\setminus \D\to \C$. As a non constant limit of univalent maps, $\phi:\C\setminus \overline \D\to \C$ is univalent and $\mj = \partial U$ with $U:=\phi (\C\setminus \overline \D)$. In particular, $\mj$ has empty interior. Since the Julia set $\mj_n$ has empty interior, the maps $\phi_n:\C\setminus \D\to \C$ are surjective, and so, the limit $\phi:\C\setminus \D\to \C$ is surjective. Indeed, if $w\in \C$ and $\phi_n(z_n)=w$, then $\phi(z)=w$ for any limit value $z$ of the sequence $\{z_n\}$. It follows that $\C = U\sqcup \mj$. Therefore, $\mj$ contains no simple closed loop since, otherwise, either $U=\com\setminus\mj$ is not connected, or $\mj$ has non-empty interior.  


We now assume $\{\beta_n\}$, $\{\beta'_n\}$ and $\{\beta''_n\}$ are sequences of points in $\mj_n$ converging to $\beta$, $\beta'$ and $\beta''$ in $\mj$.   We must prove that the sequence of nodal points of $\beta_n$, $\beta'_n$ and $\beta''_n$ in $\mj_n$ converges to the nodal point of $\beta$, $\beta'$ and $\beta''$ in $\mj$. This is based on the following lemma. 

Given two distinct points $\theta$ and $\theta'$ in $\T$, we set 
\[[\theta,\theta']_\T :=\{t\in \s^1~;~\theta,~t\text{ and }\theta'\text{ are in counterclockwise order on }\T\}.\]

\begin{lemma}\label{lem:segment}
Let $\mj\subset \C$ be a dendrite with Carathéodory loop $\varphi :\s^1\to \mj$. Let $\theta$ and $\theta'$ be two distinct points in $\s^1$. Then, 
\[\bigl[\varphi(\theta),\varphi(\theta')\bigr]_{\mj} = \varphi\bigl([\theta,\theta']_\T\bigr) \cap \varphi\bigl([\theta',\theta]_\T\bigr) .\]
\end{lemma}

\begin{proof}
Let ${\cal S}:=\bigl[\varphi(\theta),\varphi(\theta')\bigr]_{\mj}$ be the arc joining $\varphi(\theta)$ and $\varphi(\theta')$ in $\mj$. 
Since $\varphi$ is continuous, $\varphi\bigl([\theta,\theta']_\T\bigr)$ and $\varphi\bigl([\theta',\theta]_\T\bigr)$ are connected. Since they both contain $\varphi(\theta)$ and $\varphi(\theta')$, 
\[ {\cal S}\subseteq \varphi\bigl([\theta,\theta']_\T\bigr) \cap \varphi\bigl([\theta',\theta]_\T\bigr).\]
Now, assume $\xi\in  \varphi\bigl([\theta,\theta']_\T\bigr) \cap \varphi\bigl([\theta',\theta]_\T\bigr) $, i.e., $\xi=\varphi(t) = \varphi(t')$ with $t\in [\theta,\theta']_\T$ and $t'\in [\theta',\theta]_\T$. The rays of angles $t$ and $t'$ separate $\mj$   in two connected components: one contains $\varphi(\theta)$ and the other contains $\varphi(\theta')$. Thus, $\varphi(\theta)$ and $\varphi(\theta')$ are in distinct connected components of $\mj\setminus \{\xi\}$. 
Since ${\cal S}$ is an arc joining $\varphi(\theta)$ and $\varphi(\theta')$ in $\mj$, we necessarily have $\xi\in {\cal S}$. Therefore, 
\[ \varphi\bigl([\theta,\theta']_\T\bigr) \cap \varphi\bigl([\theta',\theta]_\T\bigr)\subseteq {\cal S}.\qedhere\]
\end{proof}

\begin{corollary}\label{coro:nodal}
Let $\mj\subset \C$ be a dendrite with Carathéodory loop $\varphi :\s^1\to \mj$. Let $\theta$, $\theta'$ and $\theta''$ be three distinct points in $\s^1$. Then, the nodal point of $\varphi(\theta)$, $\varphi(\theta')$ and $\varphi(\theta'')$ in $\mj$ is 
\[\varphi\bigl([\theta,\theta']_\T\bigr) \cap \varphi\bigl([\theta',\theta'']_\T\bigr)\cap \varphi\bigl([\theta'',\theta]_\T\bigr) .\]
\end{corollary}

If the three points $\beta_n$, $\beta'_n$ and $\beta''_n$ are not distinct for infinitely many $n$, then the nodal point in $\mj_n$ coincides with the corresponding multiple point, and Lemma \ref{lem:convnodal} follows easily. So, without loss of generality, assume the points are distinct and let $\theta_n$, $\theta'_n$ and $\theta''_n$ be three distinct points in $\s^1$ with $\varphi_n(\theta_n) = \beta_n$, $\varphi_n(\theta'_n) = \beta'_n$ and $\varphi_n(\theta''_n) = \beta''_n$. 
Let $\xi$ be a limit value of the sequence $\{\xi_n\}$. We must show that $\xi$ is the nodal point of $\beta$, $\beta'$ and $\beta''$ in $\mj$. 

Extracting a subsequence and reordering the points if necessary, we may assume that $\theta_n$, $\theta'_n$ and $\theta''_n$ are in counterclockwise order on $\s^1$. 
Let $\xi_n$ be the nodal points of $\beta_n$, $\beta'_n$ and $\beta''_n$ in $\mj_n$. According to Lemma \ref{lem:segment}, there are points 
\[t_n\in [\theta_n,\theta'_n]_\T,\quad t'_n\in [\theta'_n,\theta''_n]_\T\quad\text{and}\quad  t''_n\in [\theta''_n,\theta_n]_\T\]
with 
\[\xi_n = \varphi_n(t_n) = \varphi_n(t'_n) = \varphi_n(t''_n).\]
Extracting a further subsequence if necessary, we may assume that $\{\xi_n\}$, $\{\theta_n\}$, $\{\theta'_n\}$, $\{\theta''_n\}$, $\{t_n\}$, $\{t'_n\}$ and $\{t''_n\}$ converge to $\xi$, $\theta$, $\theta'$, $\theta''$, $t$, $t'$ and $t''$. Since $\{\varphi_n\}$ converges uniformly to $\varphi$, we have that 
\[\varphi(\theta) = \beta,\quad \varphi(\theta') = \beta',\quad \varphi(\theta'') = \beta''\quad\text{and}\quad \varphi(t) = \varphi(t') = \varphi(t'') = \xi.\]
 If $\theta$, $\theta'$ and $\theta''$ are not distinct, let us say $\theta = \theta'$, then $t = \theta = \theta'$ and $\xi = \beta = \beta'$ is the nodal point of $\beta$, $\beta'$ and $\beta''$ in $\mj$.
 Otherwise, if $\theta$, $\theta'$ and $\theta''$ are distinct, $t\in [\theta,\theta']_\T$, $t'\in [\theta',\theta'']_\T$ and $t''\in [\theta'',\theta]_\T$; the proof follows from  Corollary \ref{coro:nodal}.

\section{Convergence of Carathéodory loops\label{sec:convcar}}

In this section, we prove Lemma \ref{lem:convcaratheodory}. Our proof relies on puzzle techniques. 

\subsection{The puzzle of an admissible polynomial}

In the whole section, $f\in \VV$ is an admissible polynomial. We denote by $\phi_f:\C\setminus \overline\D\to \C\setminus \mj_f$ the B\"ottcher coordinate conjugating $z\mapsto z^3$ to $f$, with $\phi_f(z)/z\to 1$ as $z\to\infty$. This B\"ottcher coordinate extends as a continuous map $\phi_f:\C\setminus \D\to \C$ and the restriction to $\s^1$ is the Carathéodory loop $\varphi_f:\s^1\to \mj_f$. 

Let $\Gamma_f^0$ be the union of the equipotential $\bigl\{\phi_f(2{\rm e}^{2\pi {\rm i}\theta})~;~\theta\in \T\bigr\}$, the external rays of angles $1/4$ and $-1/4$ and their landing point $\gamma$. For $m\geq 0$, set 
\[\Gamma^m_f := f^{-m}(\Gamma_f^0).\]

\begin{definition}
The puzzle pieces of depth $m\geq 0$ are the bounded connected components of $\C\setminus \Gamma_f^m$. If $z\in \mj_f$ is not an iterated preimage of $\gamma$, we denote by $\p_m(z)$ the puzzle piece of depth $m$ which contains $z$. 
\end{definition}

\begin{figure}[htb]
\centerline{
\setlength{\unitlength}{1cm}
\begin{picture}(10,6)(0,0)
\put(0,0){{\includegraphics[width=10cm]{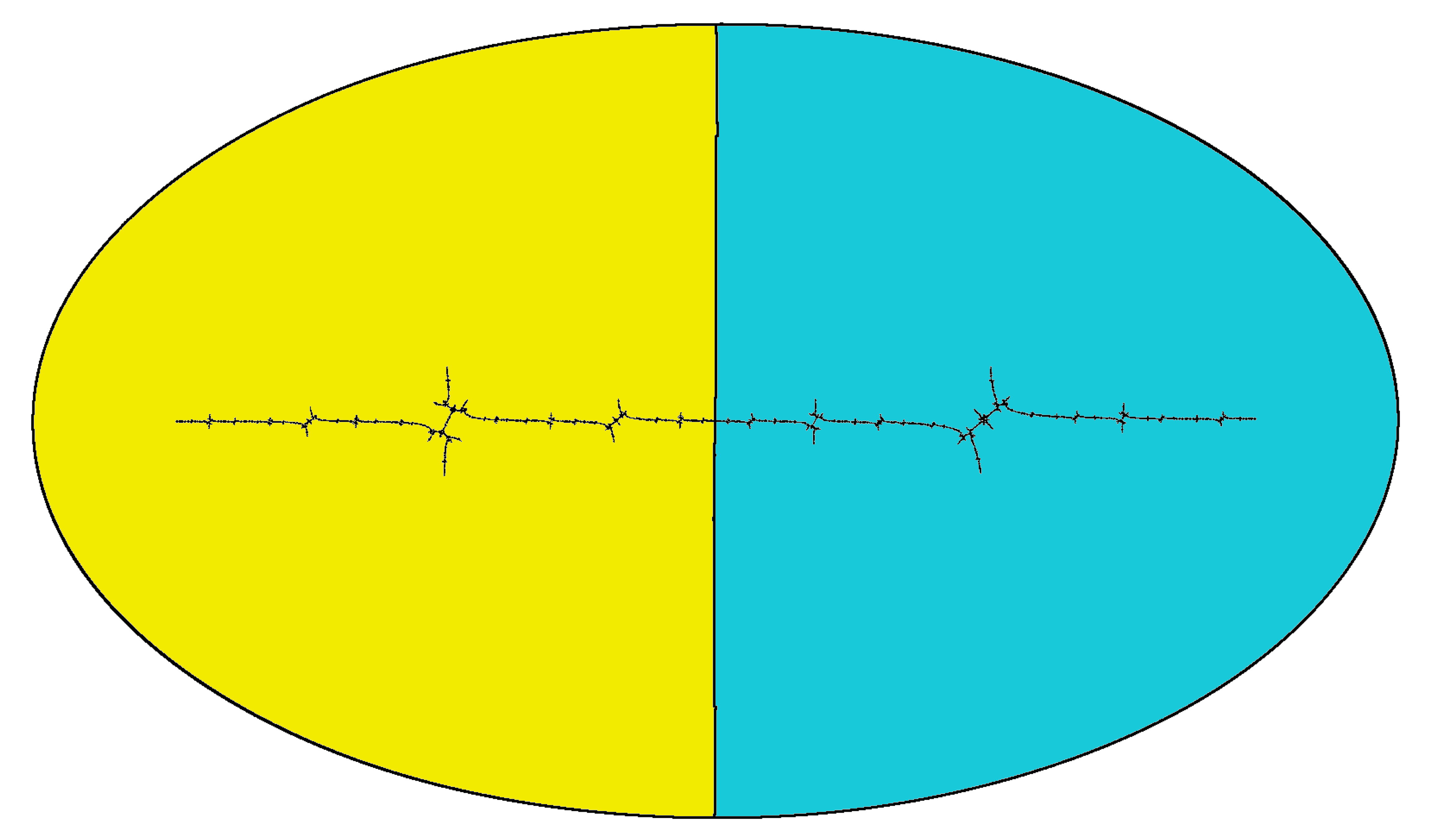}}}
\put(2,2){$\p_0(\beta)$}
\put(8,2){$\p_0(\alpha)$}
\end{picture}}
\bigskip
\centerline{
\includegraphics[width=10cm]{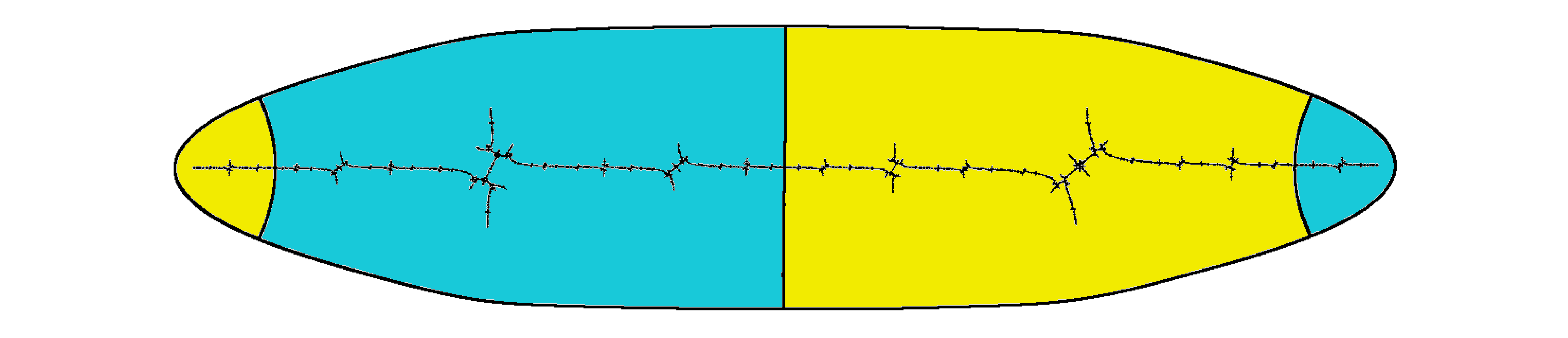}}
\bigskip
\centerline{
\includegraphics[width=10cm]{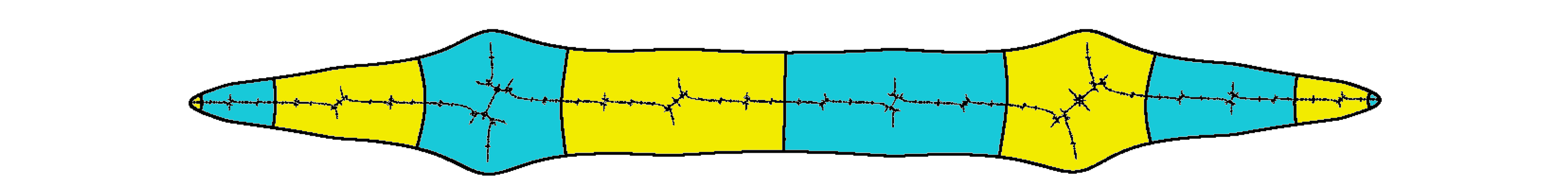}}
\bigskip
\centerline{
\includegraphics[width=10cm]{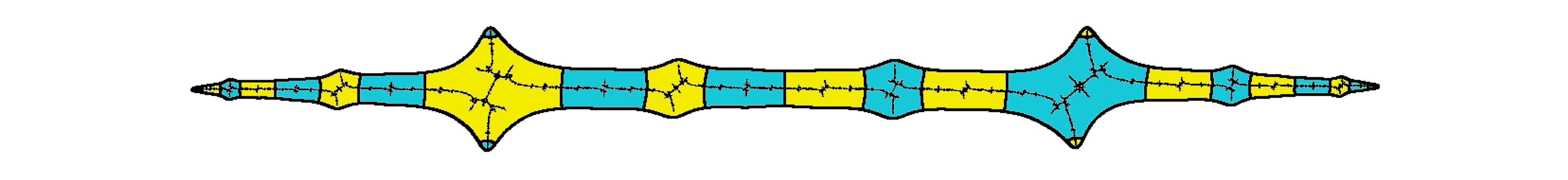}
}
\caption{The puzzle pieces of depth  $m\in \{0,1,2,3\}$, for some admissible polynomial $f\in \VV$.\label{fig:puzzle}}
\end{figure}

There are two puzzle pieces of depth $0$: one contains $\alpha$, $\omega_\alpha$ and $f(\omega_\beta)$;  the other contains $\beta$, $\omega_\beta$, $f(\omega_\alpha)$ and $f^{\circ 2}(\omega_\alpha)$. By construction, the puzzle pieces of depth $m\geq 1$ are the connected components of the preimages of the puzzle pieces of depth $m-1$. Therefore, if $\p$ is a puzzle piece of depth $m \geq 1$, then $f(\p)$ is a puzzle piece of depth $m-1$ and $f:\p\to f(\p)$ is a ramified covering. 

The union of the rays of angles $1/4$ and $-1/4$  is invariant by $f$. It follows that the puzzle pieces of depth $1$ do not intersect those rays; therefore, each puzzle piece of depth $1$ is entirely contained in a puzzle piece of depth $0$. By induction on the depth, each puzzle piece of depth $m$ is entirely contained in a puzzle piece of depth $m-1$. 

In particular, a puzzle piece contains at most one critical point. So, if $\p$ is a puzzle piece of depth $m
\geq 1$, then $f:\p\to f(\p)$ is an isomorphism if $\p$ does not contain a critical point, and a ramified covering of degree $2$ otherwise. 

%

Figure \ref{fig:puzzle} shows the puzzle pieces of depth $m\in \{0,1,2,3\}$ for some admissible polynomial $f\in \VV$. We use one color for the puzzle pieces which are iterated preimages of $\p_0(\alpha)$ and another one for the puzzle pieces which are iterated preimages of $\p_0(\beta)$.

The main result from which we deduce Lemma \ref{lem:convcaratheodory} is the following. 

\begin{proposition}\label{prop:puzzle}
The maximum diameter of a puzzle piece of depth $m$ tends to $0$ as $m$ tends to $+\infty$. 
\end{proposition}

\begin{proof}[Proof of Lemma \ref{lem:convcaratheodory} assuming Proposition \ref{prop:puzzle}]
First, if $g\in \VV$ is sufficiently close to $f$, the B\"ottcher coordinate $\phi_g$ tangent to identity at infinity is defined on the circle of radius $2$, and we may still define $\Gamma_g^0$ as the union of the equipotential $\bigl\{\phi_g(2{\rm e}^{2\pi {\rm i}\theta})~;~\theta\in \T\bigr\}$, the external rays of angles $1/4$ and $-1/4$ and their landing point $\gamma_g$. And for $m\geq 0$, we may define
\[\Gamma^m_g := g^{-m}(\Gamma_g^0).\]
Then, for each fixed $m$, $\Gamma_g^m$ depends holomorphically on $g$ in some neighborhood of $f$. 

Second, according to Proposition \ref{prop:puzzle}, given $\eps>0$, we may choose $m\geq 0$ sufficiently large so that the puzzle pieces of $f$ of depth $m$ have diameters at most $\eps/2$. Let $\UU$ be a sufficiently small neighborhood of $f$ in $\VV$ so that the puzzle of depth $m$ depends holomorphically on $g$ in $\UU$: for $g\in \UU$, there is a continuous map 
\[\psi:\UU\times \C\ni (g,z)\mapsto \psi_g(z)\in \C\] 
such that for each $g\in \UU$, $\psi_g:\C\to \C$ is a homeomorphism sending $\Gamma_f^m$ to $\Gamma_g^m$. 
Shrinking $\UU$ if necessary, we may assume that for $g\in \UU$, $d_\C(\psi_g,{\rm id})<\eps/2$. 

Third, assume $\p$ is a puzzle piece of $f$ of depth $m$, $\theta\in \T$ and $g\in \UU$ has a locally connected Julia set. Let 
$\varphi_g:\s^1\to \mj_g$ be the corresponding  Carathéodory loop. Then, 
$\varphi_f({\rm e}^{2\pi {\rm i}\theta})\in \overline \p$ if and only if  $\varphi_g({\rm e}^{2\pi {\rm i}\theta})\in \psi_g( \overline \p)$. Thus, $d_{\s^1}(\varphi_f,\varphi_g)<\eps$. 
\end{proof}

\begin{proof}[Proof of Proposition \ref{prop:puzzle}]
We will use the fact that $f$ is expanding for a suitable orbifold metric (see \cite[Th. 19.6]{Mi1}). 
Assume $k\geq 1$ and $\ell\geq 1$ are integers such that $f^{\circ k}(\omega)=\omega'$ and $f^{\circ \ell}(\omega') = \alpha$. 
Consider the function $\nu:\C\to \{1,2,4\}$ defined by
\[\nu(z):=\begin{cases} 
1&\text{if }z\notin\pf,\\
2&\text{if } z = f^{\circ n}(\omega)\text{ with } 1\leq n\leq k \text{ and }\\
4&\text{if } z = f^{\circ n}(\omega')\text{ with } 1\leq n\leq \ell.
\end{cases}
\]
The orbifold $(\C,\nu)$ has Euler characteristic
\[\chi(\C,\nu) = 1-\frac{1}{2}k- \frac{3}{4}\ell<0,\]
so that there is a universal covering of orbifold $\pi:\D\to \C$ which ramifies precisely with local degree $\nu(z)$ above $z\in \C$. 
Then, there is a metric of orbifold $\mu$ which is smooth outside $\pf$ and blows up at points in $\pf$, such that $\pi^*\mu$ is the hyperbolic metric on $\D$. There are constants $K_1$ and $K_2$ such that for all puzzle piece $\p$, 
\[\diam(\p)\leq K_1\diam_\mu(\p)\quad \text{and}\quad \diam_\mu(\p)\leq K_2,\]
where $\diam_\mu(\p)$ stands for the diameter of $\p$ for the metric $\mu$. 

In addition, there is a holomorphic map $F:\D\to \D$ such that the following diagram commutes: 
\[\diagram
\D\dto_\pi & \D\lto_F \dto^\pi\\
\C \rto^f & \C. 
\enddiagram\]
The map $F$ is contracting for the hyperbolic metric on $\D$. It follows that there is a constant $\kappa<1$ such that 
\[\forall z\in \overline{\p_0(\alpha)\cup \p_0(\beta)},\quad \|{\rm D}_z f\|_\mu >\frac{1}{\kappa}\]
(this norm may blow up at points in $f^{-1}(\pf)$). In particular, if $\p$ is a puzzle piece, then 
\[\diam_\mu(\p)\leq \begin{cases} \kappa \diam_\mu\bigl(f(\p)\bigr)&\text{if }\p\cap \cf = \emptyset\text{ and}\\
2\kappa \diam_\mu\bigl(f(\p)\bigr)&\text{if }\p\cap \cf \neq \emptyset.
\end{cases}
\]

\begin{lemma}\label{lem:piece1}
The piece $\p_1(\alpha)$ maps isomorphically to $\p_0(\alpha)$. 
\end{lemma}

\begin{proof}
Recall that the external rays of angles $1/4$ and $-1/4$ separate the plane in two connected components (see Figure \ref{fig:VV}): $U_\alpha$ containing $\alpha$ and $U_\beta$ containing $\beta$. Each ray in $U_\alpha$ has angle in $(-1/4,1/4)$. It has two preimages in $U_\beta$, one with angle in $(1/4,5/12)$ and one with angle in $(-5/12,-1/4)$, and one preimage in $U_\alpha$ with angle in $(-1/12,1/12)$. As a consequence, the component of $f^{-1}(U_\alpha)$ containing $\alpha$ is contained in $U_\alpha$ and maps isomorphically to $U_\alpha$. The other component is contained in $U_\beta$ and maps to $U_\alpha$ with  degree $2$. 

Since $\p_0(\alpha)\subset U_\alpha$, it follows that the component $\p_1(\alpha)$ of $f^{-1}\bigl(\p_0(\alpha)\bigr)$ containing $\alpha$ maps isomorphically to $\p_0(\alpha)$. 
\end{proof}

It follows by induction on $m\geq 1$ that $\p_m(\alpha)$ is the image of $\p_{m-1}(\alpha)$ by the inverse branch $f^{-1}: \p_0(\alpha)\to \p_1(\alpha)$ and that $f^{\circ m}: \p_m(\alpha)\to \p_0(\alpha)$ is an isomorphism.  As a consequence, 
\[\diam_\mu\bigl(\p_m(\alpha)\bigr)\leq K_2\cdot \kappa^m.\]
If $m$ is large enough so that $\p_m(\alpha)\cap \pf=\{\alpha\}$, then $f^{\circ \ell}:\p_{m+\ell}(\omega')\to \p_m(\alpha)$ is a ramified cover of degree $2$ and
\[\diam_\mu\bigl(\p_{m+\ell}(\omega')\bigr)\leq 2\kappa^\ell\diam_\mu\bigl(\p_{m}(\alpha)\bigr).\] It follows that there is a constant $K_3$ such that 
\[\diam_\mu\bigl(\p_m(\omega')\bigr)\leq K_3\cdot \kappa^m.\]
If $m$ is large enough so that $\p_m(\omega')\cap \pf =\{\omega'\}$, then $f^{\circ k}:\p_{m+k}(\omega)\to \p_m(\omega')$ is a ramified cover of degree $2$ and 
\[\diam_\mu\bigl(\p_{m+k}(\omega)\bigr)\leq 2\kappa^k\diam_\mu\bigl(\p_{m}(\omega)\bigr).\] 
It follows that there is a constant $K_4\geq K_3$ such that 
\[\diam_\mu\bigl(\p_m(\omega)\bigr)\leq K_4\cdot \kappa^m.\]
Finally, assume $\p$ is a piece of depth $m$ and let $n\in [0,m]$ be the least integer such that $\q:=f^{\circ n}(\p)$ contains a critical point. Then, $\q$ is a piece of depth $m-n$ and $f^{\circ n}:\p\to  \q$ is an isomorphism, so that 
\[\diam_\mu(\p)\leq \kappa^n \diam_\mu(\q) \leq \kappa^n\cdot K_4\cdot \kappa^{m-n} = K_4\cdot \kappa^m.\qedhere\]
\end{proof}

\section{The key proposition\label{sec:keyprop}}

The goal of this section is to prove the Key Proposition (Proposition~\ref{prop:key}). 

\subsection{Ray configuration}

We shall first describe some general configuration of external rays for polynomials $f\in \VV$.

\begin{lemma}\label{lem:critsep}
Assume $f\in \VV$ and assume the ray of angle $\theta$ lands on $f(\omega)$ for some critical point $\omega$. Then, the two preimage rays landing at $\omega$ separate $\alpha$ and $\beta$. One has angle in $(0,5/12)_\T$ and one has angle in $(-5/12,0)_\T$
\end{lemma}


\begin{proof}
Each ray contained in $U_\alpha$ (with angle in $(-1/4,1/4)_\T$) has a single preimage in $U_\alpha$ with angle in $(-1/12,1/12)_\T$, and two preimages in $U_\beta$, one with angle in $(1/4,5/12)_\T$ and one with angle in $(-5/12,-1/4)_\T$. So, if a ray lands at $f(\omega_\beta)\in U_\alpha$, then one preimage ray with angle in $(1/4,5/12)_\T$ and one preimage ray with angle in $(-5/12,-1/4)_\T$ land at $\omega_\beta\in U_\beta$. Those rays separate $\alpha$ and $\beta$ (see Figure~\ref{fig:critseparate}).

Similarly, if a ray lands at $f(\omega_\alpha)\in U_\beta$, then one preimage ray with angle in $(1/12,1/4)_\T$ and one preimage ray with angle in $(-1/4,-1/12)_\T$ land at $\omega_\alpha\in U_\alpha$. Those rays separate $\alpha$ and $\beta$ (see Figure~\ref{fig:critseparate}).
\end{proof}

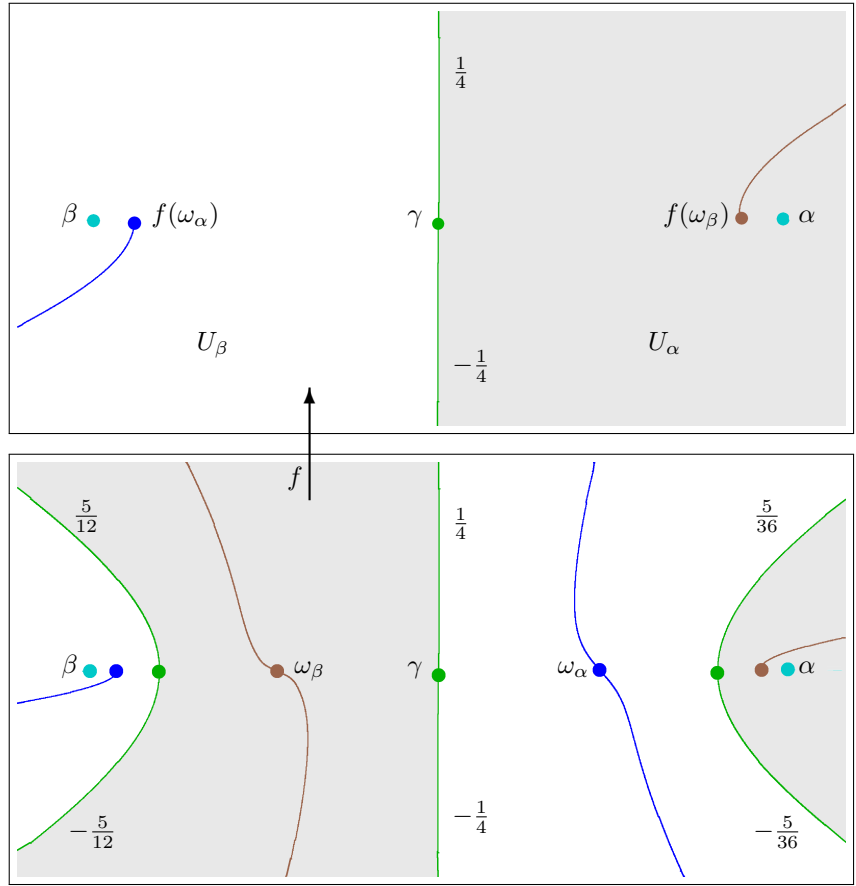
\begin{figure}[hbt!]
  \centerline{
\setlength{\unitlength}{1cm}
\begin{picture}(11,11.5)(0,0)
\put(0,0){\fbox{\includegraphics[width=11cm]{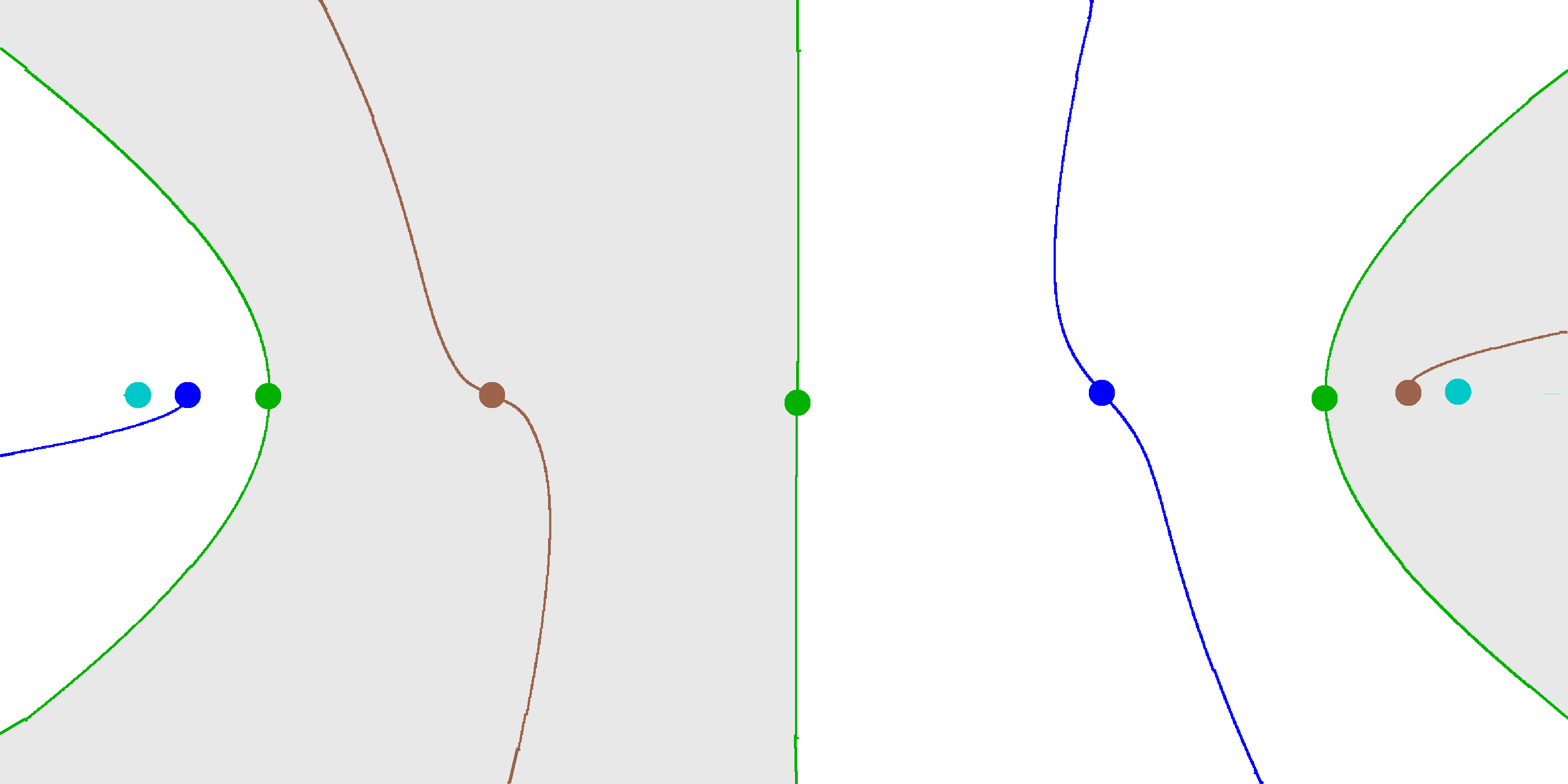}}}
\put(0,6){\fbox{\includegraphics[width=11cm]{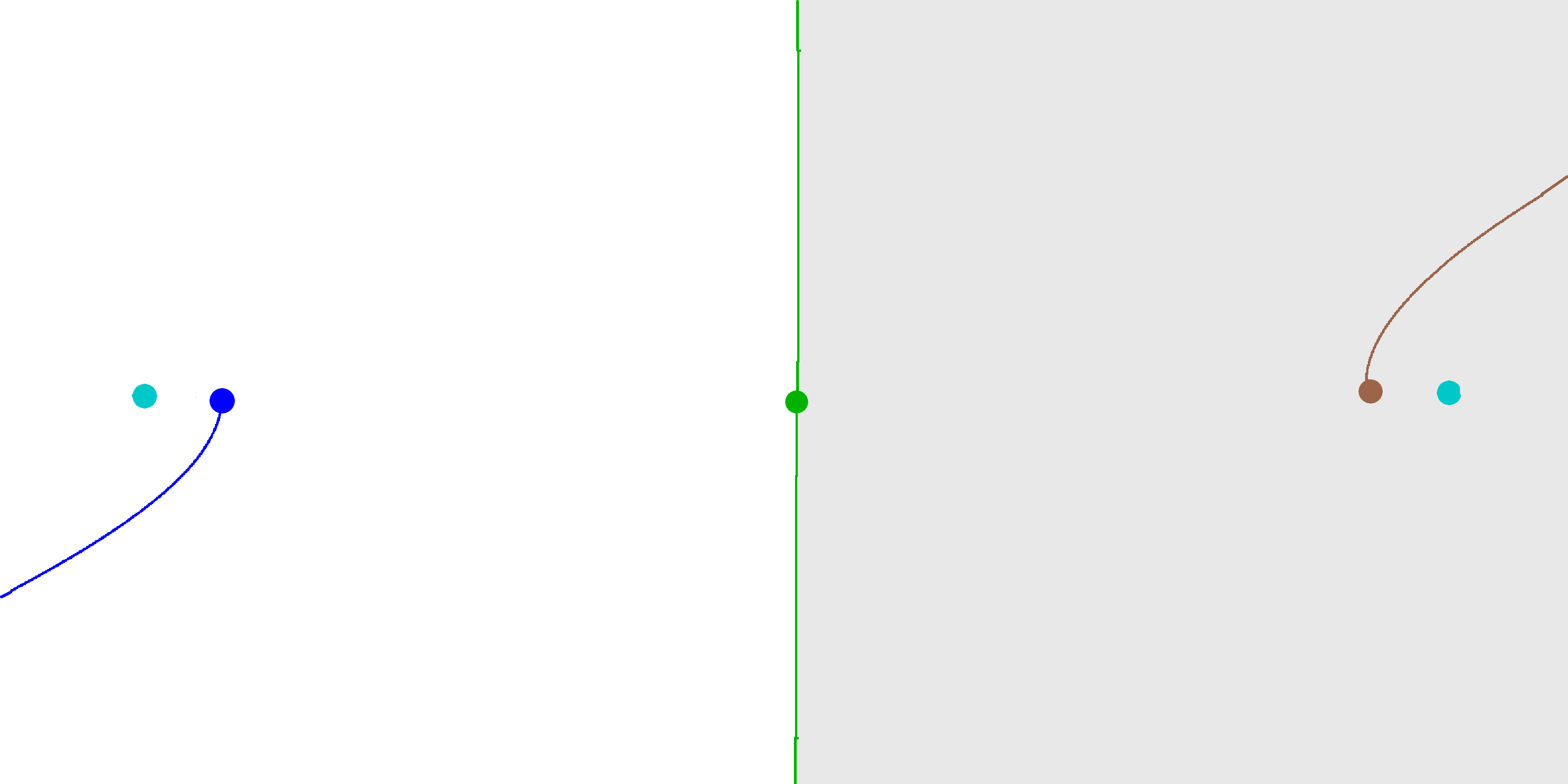}}}
\put(3.7,5.2){$f$}
\put(8.5,7){$U_\alpha$}
\put(2.5,7){$U_\beta$}
\put(5.3,2.7){$\gamma$}
\put(5.3,8.7){$\gamma$}
\put(5.9,10.6){$\frac{1}{4}$}
\put(5.9,6.7){$-\frac{1}{4}$}
\put(5.9,4.6){$\frac{1}{4}$}
\put(5.9,0.7){$-\frac{1}{4}$}
\put(10.5,2.7){$\alpha$}
\put(10.5,8.7){$\alpha$}
\put(0.7,2.7){$\beta$}
\put(0.7,8.7){$\beta$}
\put(8.7,8.7){$f(\omega_\beta)$}
\put(1.9,8.7){$f(\omega_\alpha)$}
\put(3.8,2.7){$\omega_\beta$}
\put(7.3,2.7){$\omega_\alpha$}
\put(0.795,0.5){$-\frac{5}{12}$}
\put(0.85,4.7){$\frac{5}{12}$}
\put(9.9,0.5){$-\frac{5}{36}$}
\put(9.9,4.7){$\frac{5}{36}$}
\thicklines
\put(4,5){\vector(0,1){1.5}}
\end{picture}
}
    \caption{Sketch of the configuration of external rays described in the proof of Lemma~\ref{lem:critsep}. \label{fig:critseparate}}
    
\end{figure}

\begin{lemma}\label{lem:angles}
For $f\in \VV$, the rays of angles $5/12$ and $-5/12$ land at a common point and bound an open set $W_f$ containing $\alpha$; the rays of angles $5/36$ and $-5/36$ land at a common point and bound an open set $V_f$ containing $\alpha$. Moreover, $f(V_f) = W_f$ and  $f:V_f\to W_f$ is an isomorphism. 
\end{lemma}

The result is illustrated on Figure \ref{fig:ugvg}. 
\begin{figure}[htbp]
\centerline{
\setlength{\unitlength}{1cm}
\begin{picture}(11,11.5)(0,0)
\put(0,0){\fbox{\includegraphics[width=11cm]{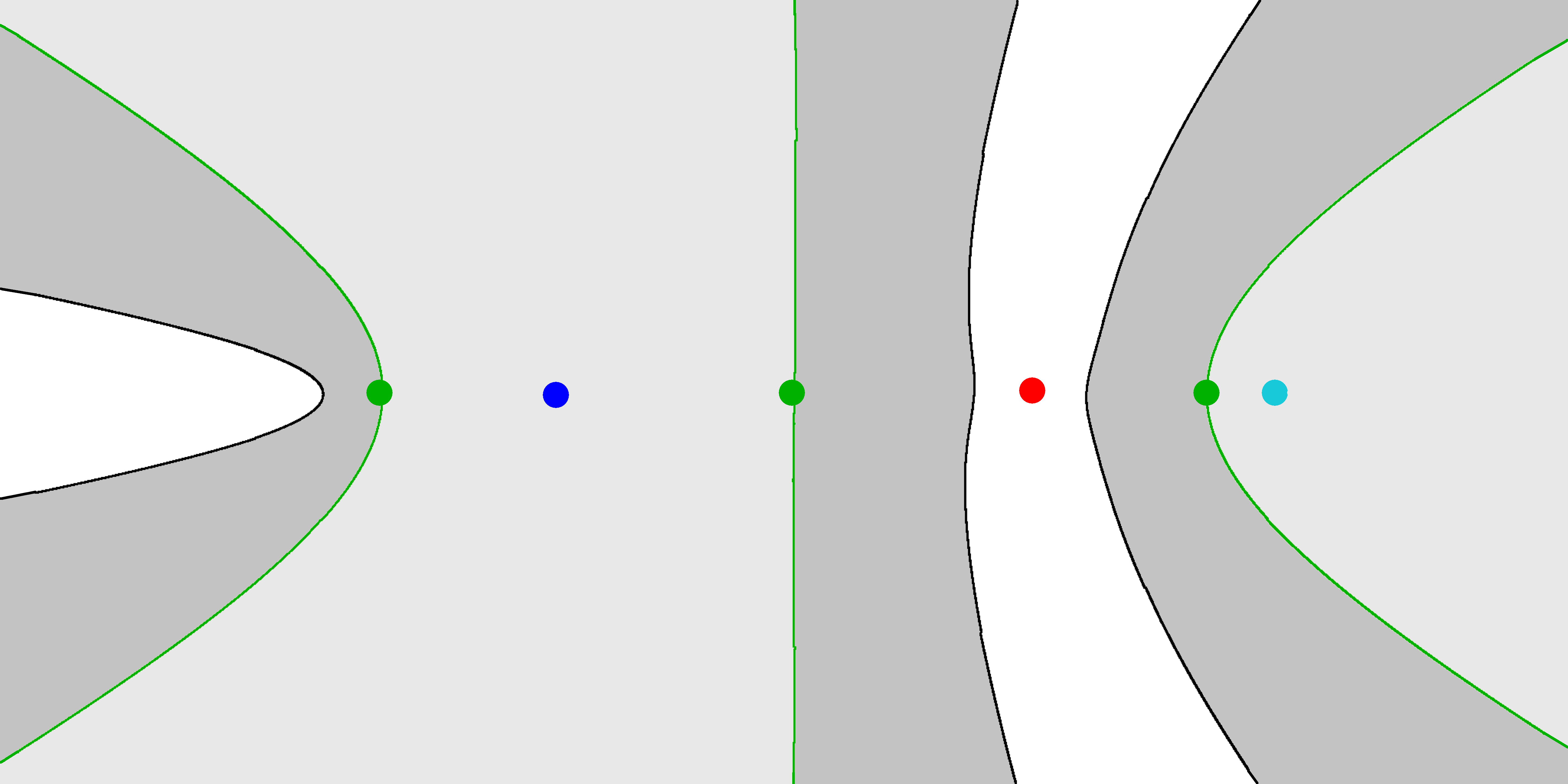}}}
\put(0,6){\fbox{\includegraphics[width=11cm]{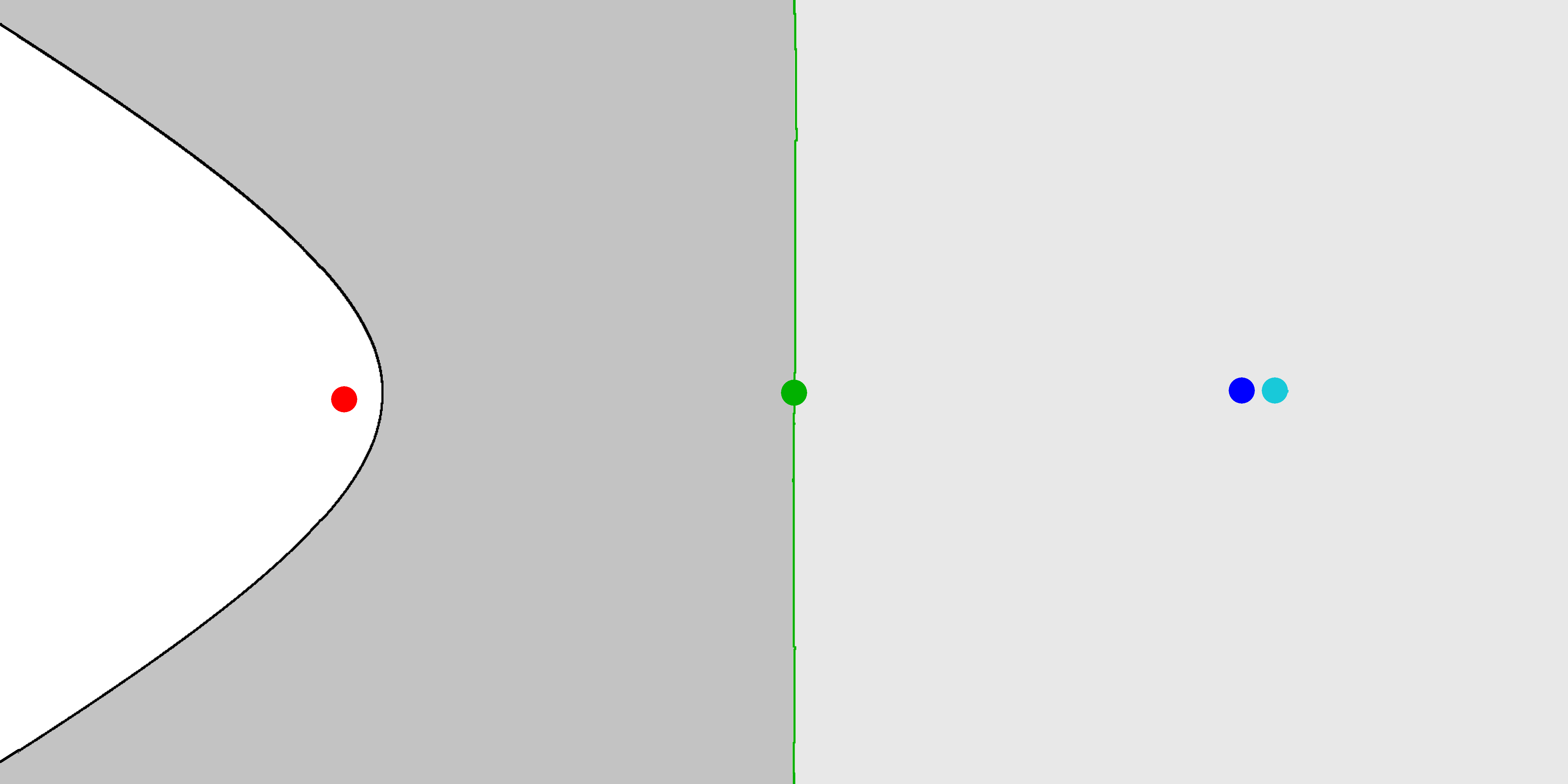}}}
\put(3.7,5.2){$f$}
\put(8.5,7){$U_\alpha$}
\put(5.5,7){$W_f$}
\put(9.7,1){$V_f$}
\put(10,2){$U'_\alpha$}
\put(.5,7.5){$U'_\beta$}
\put(5.3,2.7){$\gamma$}
\put(5.3,8.7){$\gamma$}
\put(9.3,2.7){$\alpha$}
\put(9.3,8.7){$\alpha$}
\put(7.8,8.7){$f(\omega_\beta)$}
\put(1.5,8.7){$f(\omega_\alpha)$}
\put(3.8,3){$\omega_\beta$}
\put(7.2,3){$\omega_\alpha$}
\thicklines
\put(4,5){\vector(0,1){1.5}}
\put(0.9,6.4){$-\frac{5}{12}$}
\put(1.1,10.9){$\frac{5}{12}$}
\put(8.8,0.4){$-\frac{5}{36}$}
\put(8.9,4.9){$\frac{5}{36}$}
\end{picture}
}
\caption{The open set $V_f$ containing $\alpha$ and bounded by the rays of angles $5/36$ and $-5/36$ maps isomorphically to the open set $W_f$ containing $\alpha$ and bounded by the rays of angles $5/12$ and $-5/12$. \label{fig:ugvg}}
\end{figure}

\begin{proof}
The rays of angles $1/4$ and $-1/4$ land at a common fixed point $\gamma$. This point has three distinct preimages, including itself. The rays of angles $5/12$ and $-5/12$ are contained in $U_\beta$ and have to land at common preimage of $\gamma$ contained in $U_\beta$, while the rays of angles $1/12$ and $-1/12$ are contained in $U_\alpha$ and have to land at the other preimage of $\gamma$ contained in $U_\alpha$. 

As in the proof of Lemma \ref{lem:piece1}, the component $U'_\alpha$ of $f^{-1}(U_\alpha)$ contained in $U_\alpha$ maps isomorphically to $U_\alpha$. Similarly, the component $U'_\beta$ of $f^{-1}(U_\beta)$ contained in $U_\beta$ maps isomorphically to $U_\beta$. The first is bounded by the rays of angles $1/12$ and $-1/12$ and the second is bounded by the rays of angles $5/12$ and $-5/12$. 

By assumption, $f(\omega_\alpha)\in U_\beta$ and $f^{\circ 2}(\omega_\alpha)\in U_\beta$, so that $f(\omega_\alpha)\in U'_\beta$. It follows that the region $W_f$ bounded by the rays of angle $5/12$ and $-5/12$ and containing $\alpha$ contains a single critical value $f(\omega_\beta)$. Note that $f(\omega_\beta)\in U_\alpha\subset W_f$. As a consequence, $f^{-1}(W_f)$ has two connected components. One contains $\omega_\beta$ and maps with degree $2$ to $W_f$. The other, $V_f$,  contains $U'_\alpha$ and maps isomorphically to $W_f$. This last component $V_f$ contains the ray of angle $0$ and so, is bounded by the preimages of the rays of angles $5/12$ and $-5/12$ whose angles are closest to $0$, i.e., the rays of angles $5/36$ and $-5/36$. 
\end{proof}

\subsection{Polynomials with $(k,\ell)$-configuration}\label{sec:klconfig}

Here, we assume $f\in \VV$ has $(k,\ell)$-configuration, i.e.,  
\[f^{\circ k}(\omega) = \omega'\quad\text{and}\quad f^{\circ \ell}(\omega') = \alpha\quad \text{with}\quad k\geq 1\text{ and } \ell\geq 1.\]
As $g$ varies in $\VV$ the two critical points of $g$ depend holomorphically on $g$ and $0$ remains the landing point of the ray of angle $0$. Denote by $\bo:\VV\to \C$ and $\bo':\VV\to \C$ holomorphic maps following the critical points of $g\in \VV$ with $\bo(f) = \omega$ and $\bo'(f) =\omega'$. In addition, for $m\geq 0$, let $\{\bo_{-m}:\VV\to \C\}_{m\geq 0}$ be defined recursively by 
\[\bo_0:= \bo \quad \text{and}\quad \bo_{-m-1}(g)\in V_g \quad \text{with}\quad g\bigl(\bo_{-m-1}(g)\bigr) = \bo_{-m}(g).\]
The sequence is well defined since $g:V_g\to W_g$ is an isomorphism and $V_g\subset W_g$. As $m\to +\infty$, the sequence $\{\bo_{-m}\}$ converges locally uniformly to $0$ on $\VV$.

We now exhibit a particular sequence $\{g_m\}_{m\geq m_0}$ of polynomials converging to $f$ in $\VV$. 

\begin{lemma}\label{lem:pm}
Assume $f\in \VV$ has $(k,\ell)$-configuration. Then, there is a sequence $\{g_m\}_{m\geq m_0}$ converging to $f$ in $\VV$ such that $g_m$ has $(m+\ell,k+\ell)$-configuration with critical points $\omega_m$ and $\omega'_m$ satisfying:
\[ \omega_m=\bo'(g_m),\quad \omega'_m:=\bo(g_m)\quad \text{and}\quad g_m^{\circ \ell}(\omega_m) =\bo_{-m}(g_m).\] 
\end{lemma}

\begin{remark}
The roles of the two critical points are exchanged: the sequence $\{\omega_m\}$ converges to $\omega'$ and the sequence $\{\omega'_m\}$ converges to $\omega_m$ (see Figure~\ref{fig:perturbadmiss}).  
\end{remark}

    \begin{figure}[hbt!]
    \centering
    \subfigure[\scriptsize{$f$
    }  ]{
    \hspace{-1.1cm}
	\setlength{\unitlength}{300pt}
	\begin{picture}(1,0.14620103)%
    \put(0,0){\includegraphics[width=\unitlength]{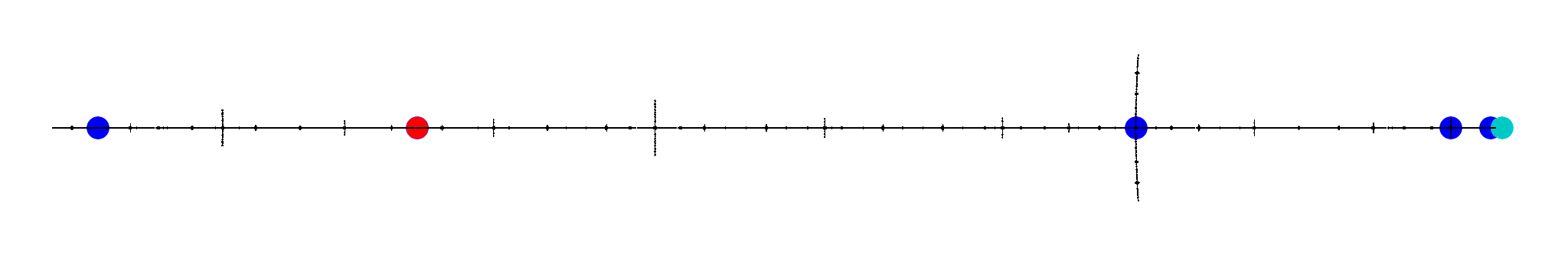}}%
    \put(0.87298389,0.09266568){\color[rgb]{0,0,0}\makebox(0,0)[lb]{\smash{\textcolor{blue}{$\omega_{-1}$}}}}%
    \put(0.95451379,0.09247511){\color[rgb]{0,0,0}\makebox(0,0)[lb]{\smash{$\alpha=f(\omega')$}}}%
    \put(0.73439765,0.0869192){\color[rgb]{0,0,0}\makebox(0,0)[lb]{\smash{\textcolor{blue}{$\omega$}}}}%
    \put(0.23669652,0.09217759){\color[rgb]{0,0,0}\makebox(0,0)[lb]{\smash{\textcolor{red}{$\omega'=f^{\circ 2}(\omega)$}}}}%
    \put(0.05102795,0.09317759){\color[rgb]{0,0,0}\makebox(0,0)[lb]{\smash{\textcolor{blue}{$f(\omega')$}}}}%
    \put(0.93872724,0.05170341){\color[rgb]{0,0,0}\makebox(0,0)[lb]{\smash{\textcolor{blue}{$\omega_{-2}$}}}}
     \end{picture}     
   }
    \hspace{0.1in}
\subfigure[\scriptsize{$g_1$}  ]{
\hspace{-1.1cm}
\setlength{\unitlength}{300pt}
\begin{picture}(1,0.14696772)%
    \put(0,0){\includegraphics[width=\unitlength]{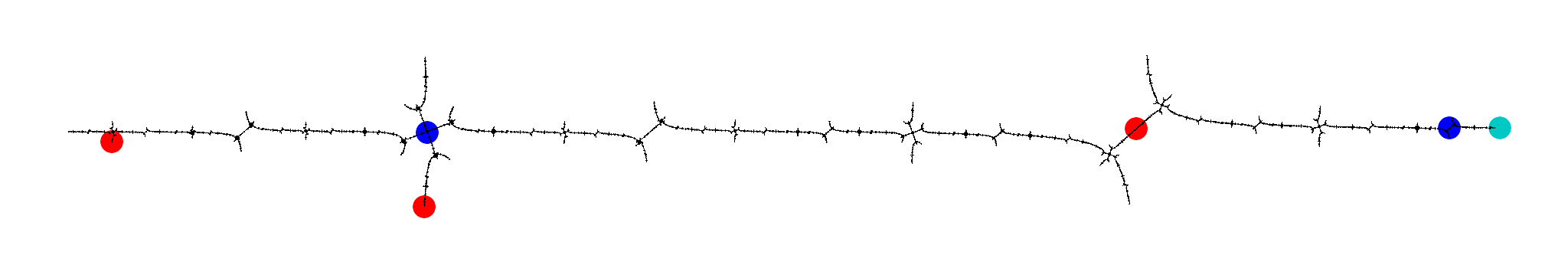}}%
    \put(0.05917082,0.0392453){\color[rgb]{0,0,0}\makebox(0,0)[lb]{\smash{\textcolor{red}{$g_1(\omega'_1)$}}}}%
    \put(0.22483959,0.08445751){\color[rgb]{0,0,0}\makebox(0,0)[lb]{\smash{\textcolor{blue}{$\omega_1$}}}}%
    \put(0.28359097,0.01313623){\color[rgb]{0,0,0}\makebox(0,0)[lb]{\smash{\textcolor{red}{$g_1^{\circ 2}(\omega'_1)$}}}}%
    \put(0.54180999,0.09699652){\color[rgb]{0,0,0}\makebox(0,0)[lb]{\smash{\textcolor{red}{$g_1^{\circ 2}(\omega_1)=\omega'_1$}}}}%
    \put(0.95117511,0.05285753){\color[rgb]{0,0,0}\makebox(0,0)[lb]{\smash{$\alpha=g_1^{\circ 3}(\omega'_1)$}}}%
    \put(0.85607436,0.0950744){\color[rgb]{0,0,0}\makebox(0,0)[lb]{\smash{\textcolor{blue}{$g_1(\omega_1)$}}}}%
  \end{picture}
	}   
   
    \hspace{0.1in}
    \subfigure[\scriptsize{$g_2$
     }  ]{
     \hspace{-1.1cm}
    \setlength{\unitlength}{300pt}
    \begin{picture}(1,0.14526769)%
    \put(0,0){\includegraphics[width=\unitlength]{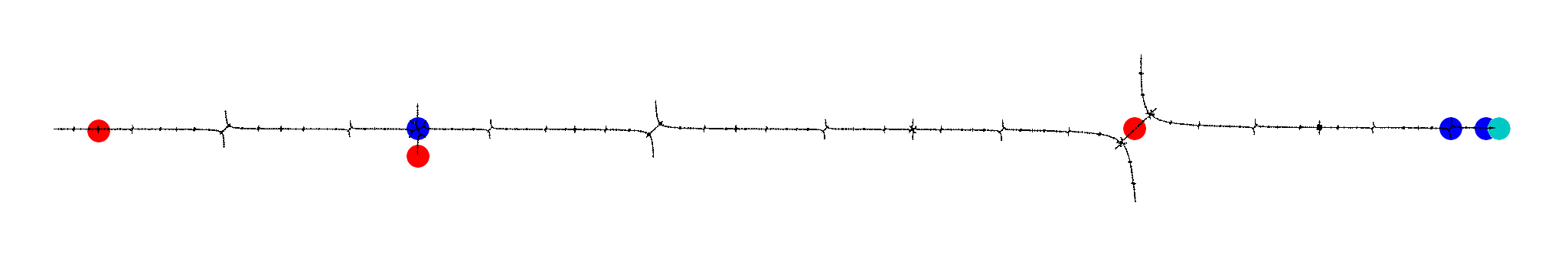}}%
    \put(0.93745532,0.04578948){\color[rgb]{0,0,0}\makebox(0,0)[lb]{\smash{\textcolor{blue}{$g_2(\omega_2)$}}}}%
    \put(0.82593213,0.09534626){\color[rgb]{0,0,0}\makebox(0,0)[lb]{\smash{\textcolor{blue}{$g_2^{\circ 2}(\omega_2)$}}}}%
    \put(0.94940698,0.09252583){\color[rgb]{0,0,0}\makebox(0,0)[lb]{\smash{$\alpha=g_2^{\circ 3}(\omega'_2)$}}}%
    \put(0.53001382,0.09508953){\color[rgb]{0,0,0}\makebox(0,0)[lb]{\smash{\textcolor{red}{$g_2^{\circ 3}(\omega_2)=\omega'_2$}}}}%
    \put(0.05507775,0.04020561){\color[rgb]{0,0,0}\makebox(0,0)[lb]{\smash{\textcolor{red}{$g_2(\omega'_2)$}}}}%
    \put(0.25385235,0.0283399){\color[rgb]{0,0,0}\makebox(0,0)[lb]{\smash{\textcolor{red}{$g_2^{\circ 2}(\omega'_2)$}}}}%
    \put(0.22785231,0.09115836){\color[rgb]{0,0,0}\makebox(0,0)[lb]{\smash{\textcolor{blue}{$\omega_2$}}}}%
  \end{picture}  
    }

    \caption{An admissible polynomial $f$ with a $(2,1)$-configur\-ation together with two perturbations $g_1$ and $g_2$.}
    \label{fig:perturbadmiss}
    \end{figure}
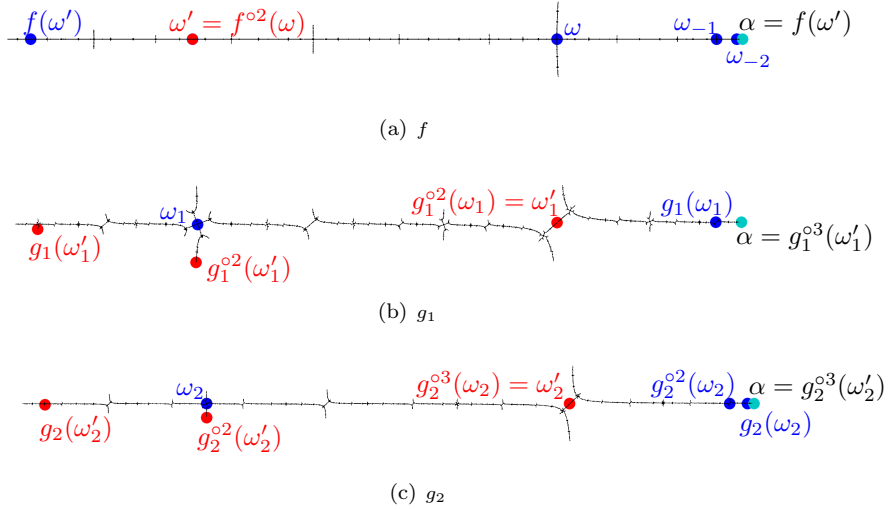

\begin{proof}
For $(a,b)\in \C^2$, let $f_{a,b}\in \PP$ be the cubic polynomial defined by 
\[f_{a,b}(z) = z^3 - \frac{3}{2} (a+b) z^2 + 3ab z.\]
The critical points of $f_{a,b}$ are $a$ and $b$, so that $f = f_{\omega,\omega'}$. 

Consider the analytic sets 
\[\Sigma:=\left\{(a,b)\in \C^2~;~ f_{a,b}^{\circ (k+\ell)}(a) = 0\right\}\quad \text{and}\quad
\Sigma':=\left\{(a,b)\in \C^2~;~ f_{a,b}^{\circ \ell}(b) = 0\right\}.\] 
There are parameters $(a,b)$ which belong neither to $\Sigma$ nor to $\Sigma'$ (for example, when $a := {\rm i}\sqrt{2}/2$ and $b = -{\rm i}\sqrt{2}/2$, then  $a$ is a fixed point of $f_{a,b}$). As a consequence, $\Sigma$ and $\Sigma'$ are $1$-dimensional complex curves. 

By assumption, $(\omega,\omega')\in \Sigma\cap \Sigma'$. The intersection $\Sigma\cap \Sigma'$ consists of postcritically finite polynomials, thus is bounded in $\C^2$. It follows that $(\omega,\omega')$ is an isolated point of $\Sigma\cap \Sigma'$. 
Let $\a:(\D,0)\to (\C,\omega)$ and $\b:(\D,0)\to (\C,\omega')$ be non constant analytic germs so that $\bigl(\a(t),\b(t)\bigr)\in \Sigma\cap \VV$ for $t\in \D$. Set $F_t:=f_{\a(t),\b(t)}\in \VV$ and consider the sequence of functions $\{\sigma_m:\D\to \C\}_{m\geq 0}$ defined by 
\[\sigma_m(t) := F_t^{\circ \ell}\bigl(\b(t)\bigr) - \bo_{-m}(F_t).\]
As $m$ tends to $+\infty$, the sequence $\{\sigma_m\}$ converges to $\sigma:\D\ni t\mapsto F_t^{\circ \ell}\bigl(\b(t)\bigr)\in \C$. Note that $\sigma$ vanishes at $0$ but does not identically vanish since otherwise, the curve $t\mapsto F_t$ would take its values in $\Sigma\cap \Sigma'$, contradicting the previous observation that $(\omega,\omega')$ is an isolated point of $\Sigma\cap \Sigma'$. It follows from the Rouch\'e Theorem that for $m$ large enough, $\sigma_m$ vanishes at some point $t_m\in \D$ with $t_m\to 0$ as $m\to +\infty$. The result follows with $g_m:=F_{t_m}$, $\omega_m:=\b(t_m)$ and $\omega'_m:=\a(t_m)$. 
\end{proof}

\subsection{Admissible perturbations}

Here, we prove the Key Proposition (Proposition~\ref{prop:key}). Its proof follows directly from Lemma~\ref{lem:profkeyprop}. 

We assume $f\in \VV$ is admissible, $\beta$ is the landing point of the ray of angle $1/2$ for $f$, $\xi$ is the nodal point of $\beta$, $\beta'$ and $\beta''$ in $\mj_f$ with $f^{-1}(\beta) = \{\beta, \beta',\beta''\}$, and  
\[f^{\circ j}(\xi) = \omega, \quad f^{\circ k}(\omega) = \omega'\quad \text{and}\quad f^{\circ \ell}(\omega') = \alpha.\]
Let $\{g_m\}_{m\geq m_0}$ be a sequence of cubic polynomials provided by Lemma \ref{lem:pm}. For $m\geq m_0$, let $\beta_m$ be the landing point of the ray of angle $1/2$ for $g_m$, and let $\xi_m$ be the nodal point of $\beta_m$, $\beta'_m$ and $\beta''_m$ in $\mj_{g_m}$ with $g_m^{-1}(\beta_m) = \{\beta_m, \beta'_m,\beta''_m\}$.

\begin{lemma}\label{lem:profkeyprop}
If $m$ is large enough, then $g_m^{\circ (j+k)}(\xi_m)=\omega_m$. 
\end{lemma}

The idea of the proof is the following. For any polynomial in $\VV$, the fixed point $\alpha=0$ is the landing point of a unique external ray (of angle  $0$). If the polynomial has $(k,\ell)$-configuration, then $f^{\circ \ell}$ sends the critical point $\omega'$ to $\alpha$ with local degree $2$ and $f^{\circ k}$ sends the critical point $\omega$ to $\omega'$ with local degree $2$. It follows that $\omega'$ is the landing point of exactly two external rays and $\omega$ is the landing point of exactly four external rays. 

If in addition $f$ is admissible, then $f^{\circ j}$ sends $\xi$ to $\omega$ with local degree $1$,  and $\xi$ is the landing point of exactly four external rays separating $\beta$, $\beta'$ and $\beta''$. Note that $f^{\circ (j+k)}$ has a critical point at $\xi$ with critical value $\omega'$. 

For the perturbed map $g_m$,  the critical point $\omega_m$ is the landing of exactly  four external rays. The map $g_m^{\circ (j+k)}$ has a critical point close to $\xi$ with critical value close to $\omega_m$, but different from $\omega_m$. It follows that there are two points $\xi_m^\pm$ close to $\xi$ which are mapped to $\omega_m$ by $g_m^{\circ (j+k)}$. Exactly four rays land at each of these two points. We shall see that those eight rays converge to the four rays landing at $\xi$ for $f$, and that one of the two points $\xi_m^\pm$ separates $\beta_m$, $\beta'_m$ and $\beta''_m$ in $\mj_{g_m}$.

\begin{proof}
Note that $f^{\circ j}:(\C,\xi)\to (\C,\omega)$ has local degree $1$ at $\xi$, $f^{\circ k}:(\C,\omega)\to (\C,\omega')$ has local degree $2$ at $\omega$ and $f^{\circ \ell}:(\C,\omega')\to (\C,\alpha)$ has local degree $2$ at $\omega'$. 
Let $D\Subset \widehat D$ be sufficiently small disks around $\alpha=0$ so that $\pf\cap \widehat D = \{0\}$. Let $\widehat D'$ be the component of $f^{-\ell}(\widehat D)$ which contains $\omega'$, and let $\widehat D''$ be the component of $f^{-j-k}(\widehat D')$ which contains $\xi$. For $m$ large enough, 
\begin{itemize}
\item $g_m^{-\ell}(D)$ has a component $D'_m\Subset \widehat D'$ containing $\omega_m$ and $g_m^{\circ k}(\omega'_m)$,  
\item  $g_m^{-j-k}(D'_m)$ has a component $D''_m\Subset \widehat D''$ containing a point $\zeta_m$ and two points $\xi_m^\pm$ such that 
\[g_m^{\circ j}(\zeta_m) = \omega'_m\quad\text{and}\quad g_m^{\circ (j+k)}(\xi_m^+) = g_m^{\circ (j+k)}(\xi_m^-) = \omega_m.\] 
\end{itemize}
As $m$ tends to $+\infty$, the sequence $\{\zeta_m\}$ and $\{\xi_m^\pm\}$ converge to $\xi$. A summary of the dynamics within the preimages of $D$ before and after perturbation is provided in Figure~\ref{fig:preimalpha}.

\begin{figure}[hbt!]
    \centering
    \subfigure{
     \def\svgwidth{350pt}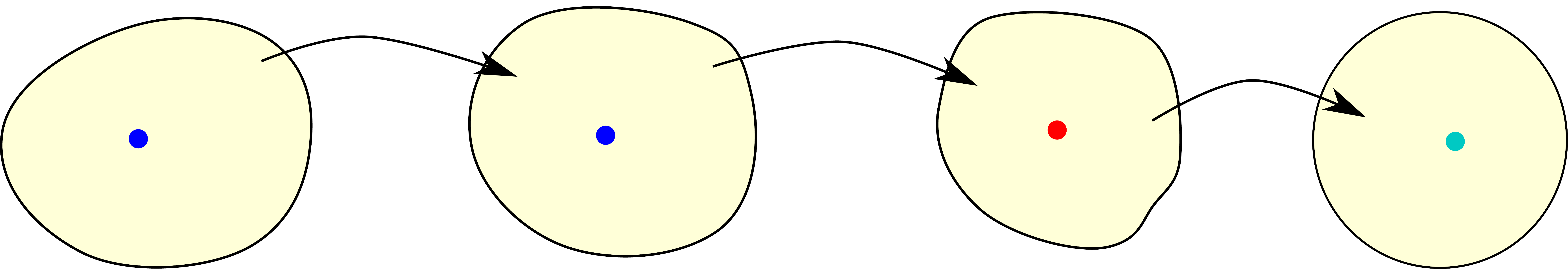 
   }
    \vspace{0.1in}
    \subfigure{
   \def\svgwidth{350pt}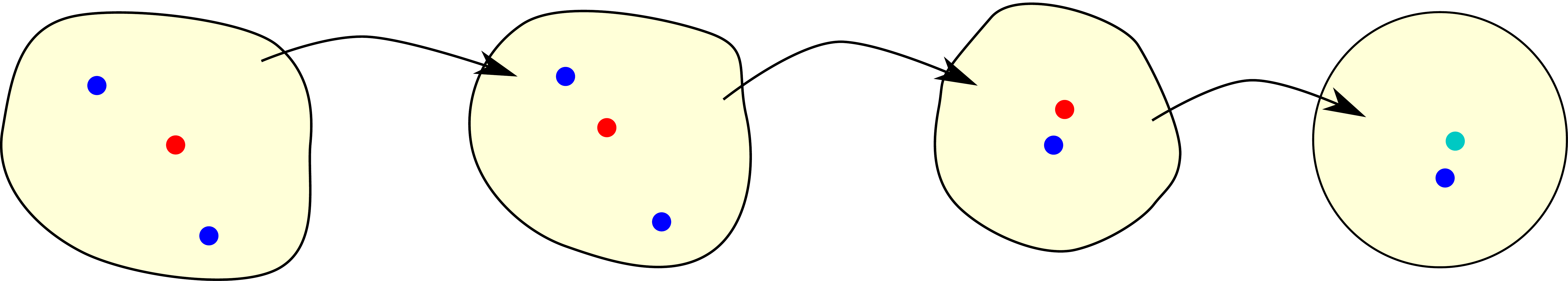 
    }
    \caption{Dynamics within the preimages of the disk $D$ containing $\alpha$ before and after perturbation. \label{fig:preimalpha}}
    
\end{figure}

Given $\theta\in \T$, denote by $R(\theta)$ the ray of angle $\theta$ for $f$, and by $R_m(\theta)$ the ray of angle $\theta$ for $g_m$. 
There is a single ray landing at $\alpha$: $R(0)$. So, there are two rays landing at $\omega'$, four rays landing at $\omega$ and four rays landing at $\xi$. Let $\theta_1$, $\theta_2$, $\theta_3$ and $\theta_4$ be the angles of the four rays landing at $\xi$, cyclically ordered counterclockwise. Then, modulo $1$, we have that $3^{j+k}\theta_1 = 3^{j+k}\theta_3=:\eta_1$ and $3^{j+k}\theta_2 = 3^{j+k}\theta_4=:\eta_2$, and the rays $R(\eta_1)$ and $R(\eta_2)$ land at $\omega'$. In addition, modulo $1$,  we have that $3^\ell \eta_1 = 3^\ell\eta_2 =  0$ and $R(0)$ lands at $\alpha=0$.

Since $g_m^{\circ \ell}(\omega_m)\neq 0$ and $R_m(0)$ lands at $\alpha=0$, for $m$ large enough, the rays $R_m(\eta_1)$ and $R_m(\eta_2)$ land at two distinct points in $D'_m$. Since $g_m^{\circ (k+\ell)}(\omega'_m) = 0$, one of those rays lands at $g_m^{\circ k}(\omega'_m)$. Without loss of generality, relabelling the rays if necessary, we may assume that this ray is $R_m(\eta_1)$. Then, $R_m(\theta_1)$ and $R_m(\theta_3)$ land at $\zeta_m$, whereas $R_m(\theta_2)$ and $R_m(\theta_4)$ land at two distinct points in $D''_m$. Note that the four rays $R(\theta_1)$, $R(\theta_2)$, $R(\theta_3)$ and $R(\theta_4)$ land at $\xi$ and separate the plane in four connected components. The points $\beta$, $\beta'$ and $\beta''$ are in $3$ distinct connected components. So, $1/2$, $1/6$ and $-1/6$ must belong to distinct components of $\T\setminus \{\theta_1,\theta_2,\theta_3,\theta_4\}$. Without loss of generality, relabelling the rays if necessary, we may assume that one of the angles $1/2$, $1/6$ and $-1/6$ belongs to $(\theta_1,\theta_2)_\T$, one belongs to $(\theta_2,\theta_3)_\T$ and one belongs to $(\theta_3,\theta_1)_\T$  (see Figure~\ref{fig:rayspert}). 

According to Lemma \ref{lem:critsep}, the two rays landing at $\omega'_m$ separate $\alpha=0$ and $\beta_m$. One has angle in $(0,5/12)_\T$ and the other has angle in $(-5/12,0)_\T$ (see Lemma \ref{lem:critsep}). Let us recall that  $g_m^{\circ \ell}(\omega_m)=\bo_{-m}(g_m)$ is the $m$-th iterated preimage of $\omega'_m$ by the univalent branch $g_m:V_{g_m}\to W_{g_m}$. 
It follows from Lemma~\ref{lem:angles} that, if a ray in $W_{g_m}$ has angle in $(0,5/12)_\T$ (respectively $(-5/12,0)_\T$), then the preimage ray in $V_{g_m}$ has angle in $(0,5/36)_\T$ (respectively $(-5/36,0)_\T$). Consequently, the rays landing at $g_m^{\circ \ell}(\omega_m)$ have angles with representatives
\[\eps_m^+\in \left(0,\frac{5}{12\cdot 3^m}\right)\quad\text{and}\quad \eps_m^-\in \left(-\frac{5}{12\cdot 3^m},0\right).\]
 They separate $\alpha=0$ and $\beta_m$.

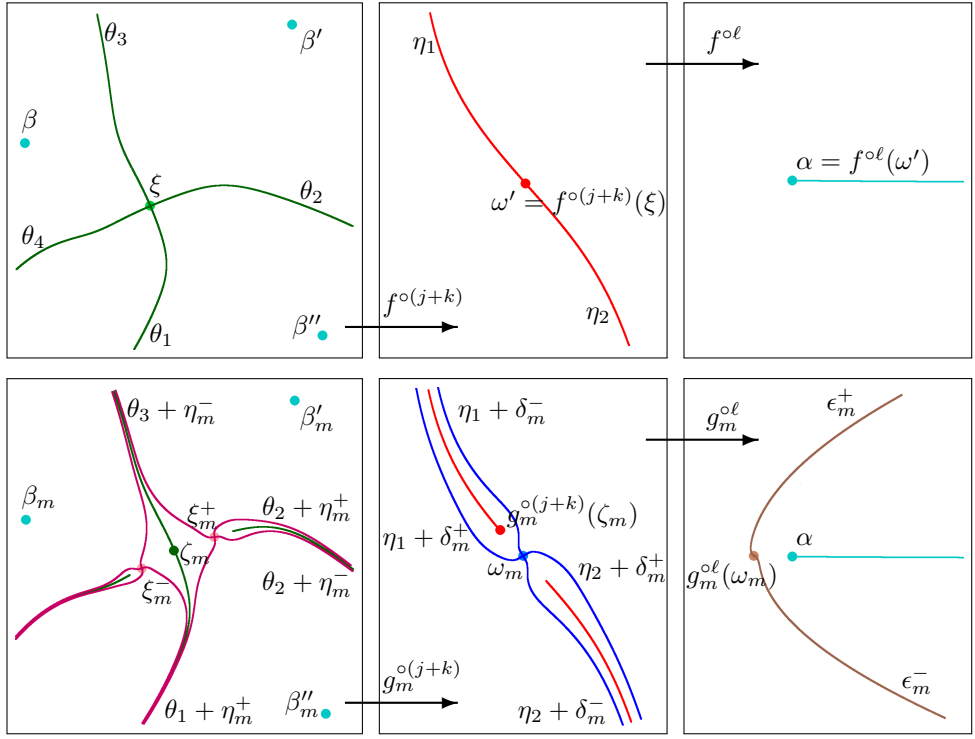
\begin{figure}[hbt!]
  \centerline{
\setlength{\unitlength}{1cm}
\begin{picture}(13,10)(0,0)
\put(0,5){\fbox{\includegraphics[width=4.5cm]{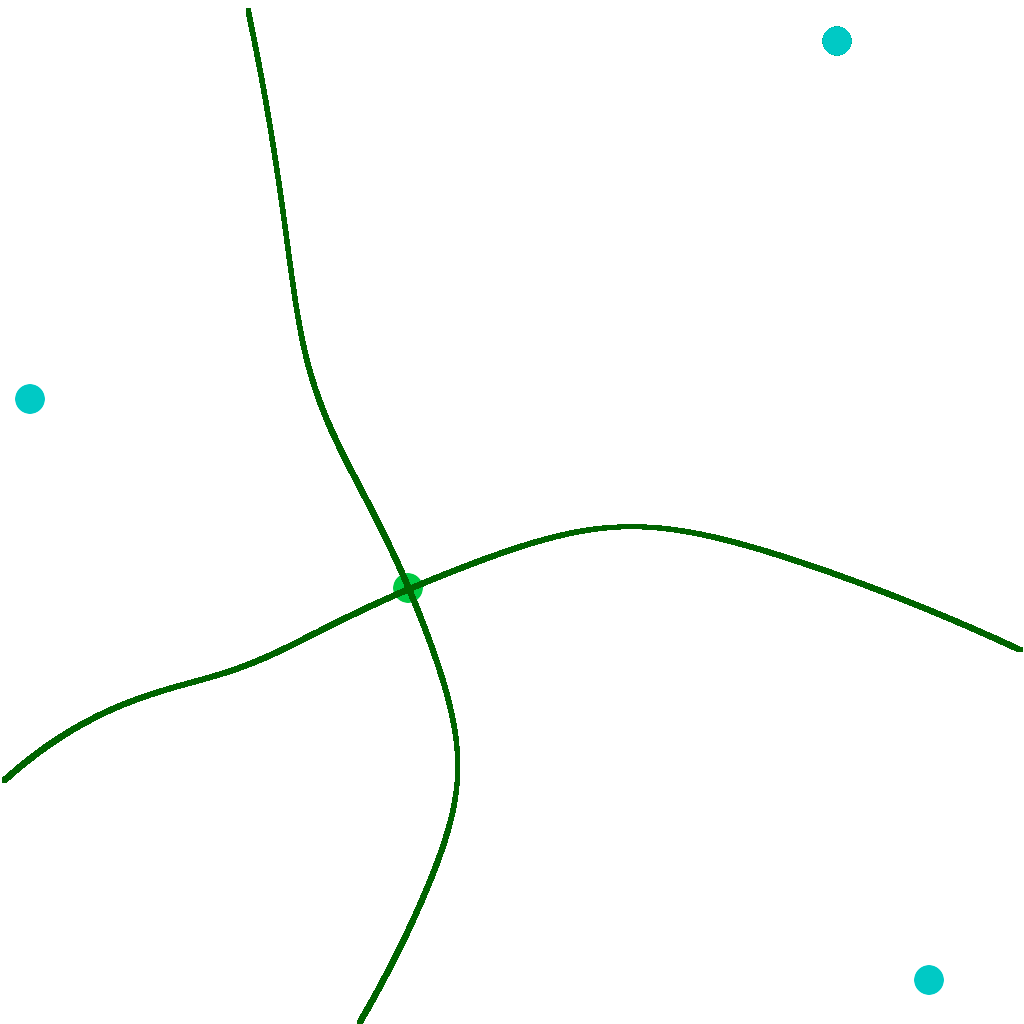}}}
\put(4.95,5){\fbox{\includegraphics[width=3.6cm]{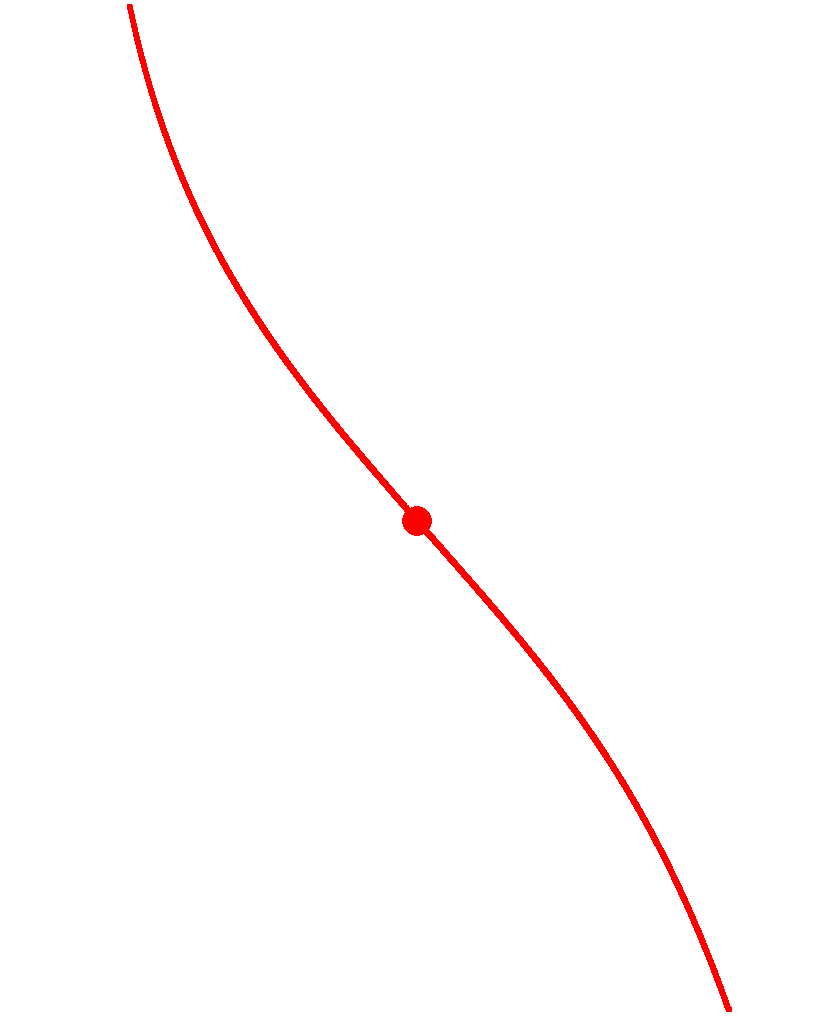}}}
\put(9,5){\fbox{\includegraphics[width=3.6cm]{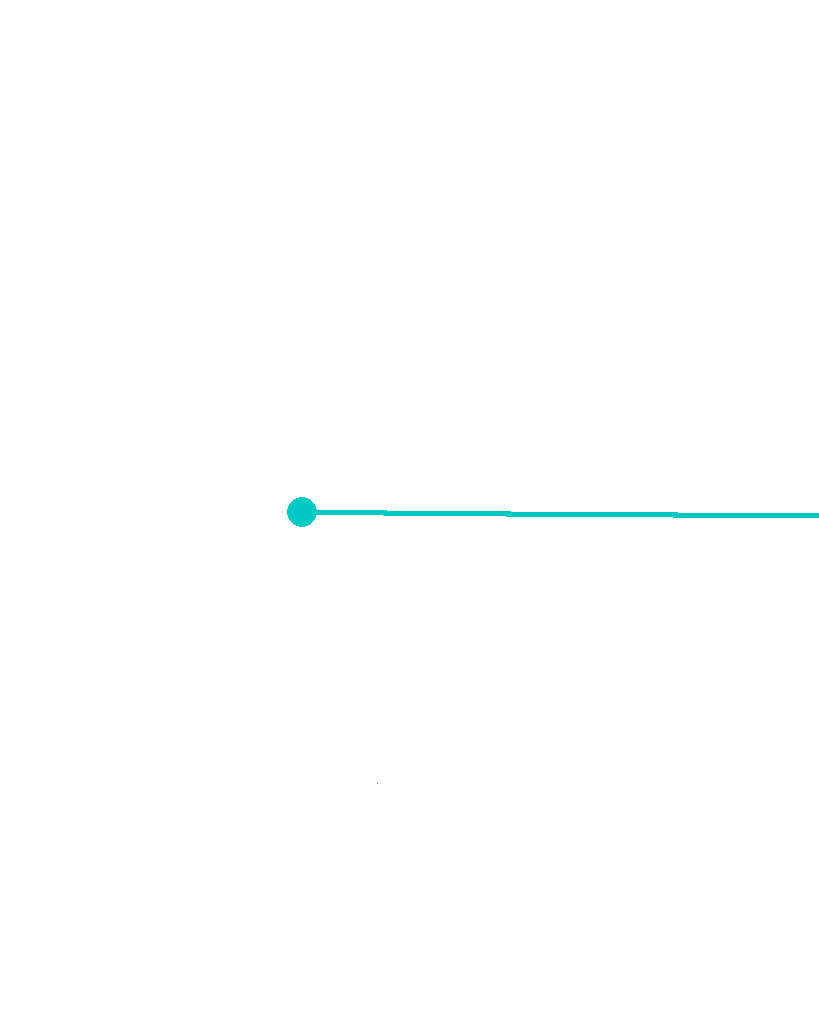}}}
\put(0,0){\fbox{\includegraphics[width=4.5cm]{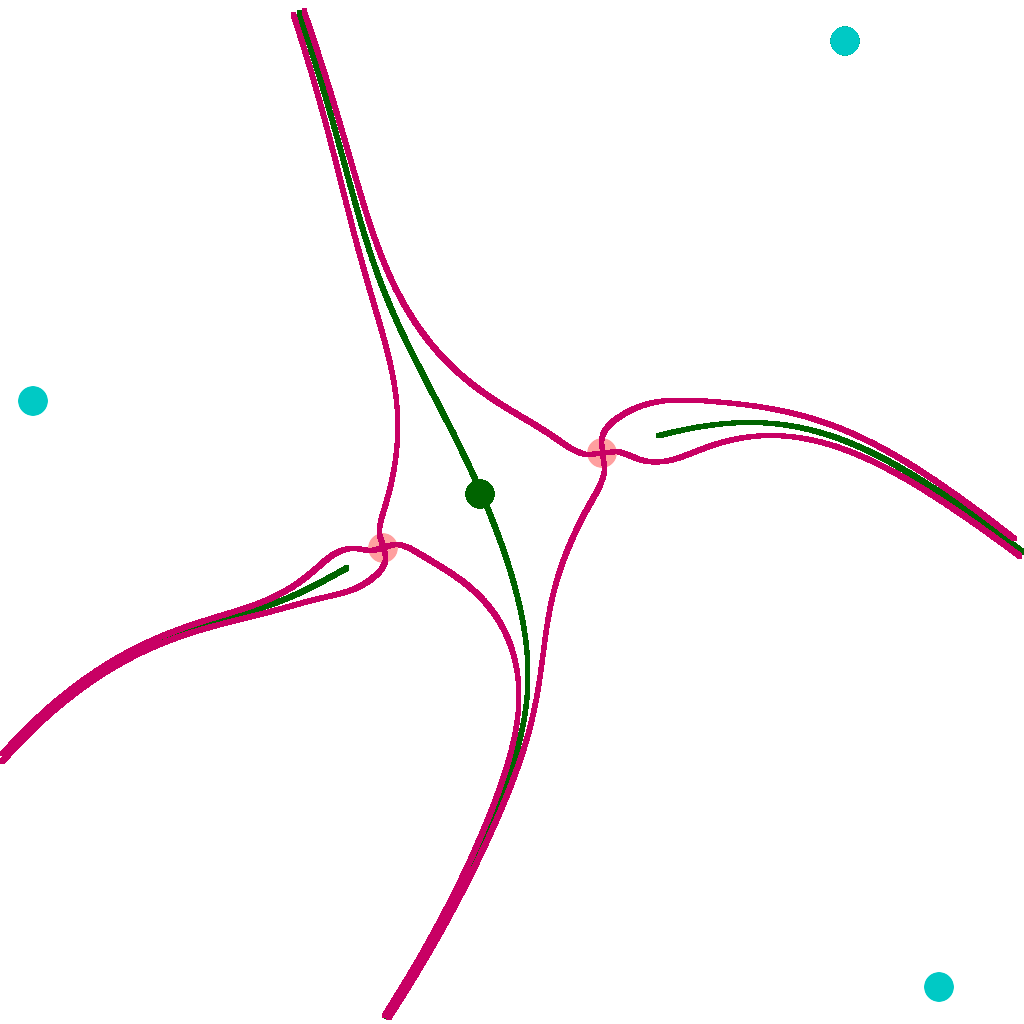}}}
\put(4.95,0){\fbox{\includegraphics[width=3.6cm]{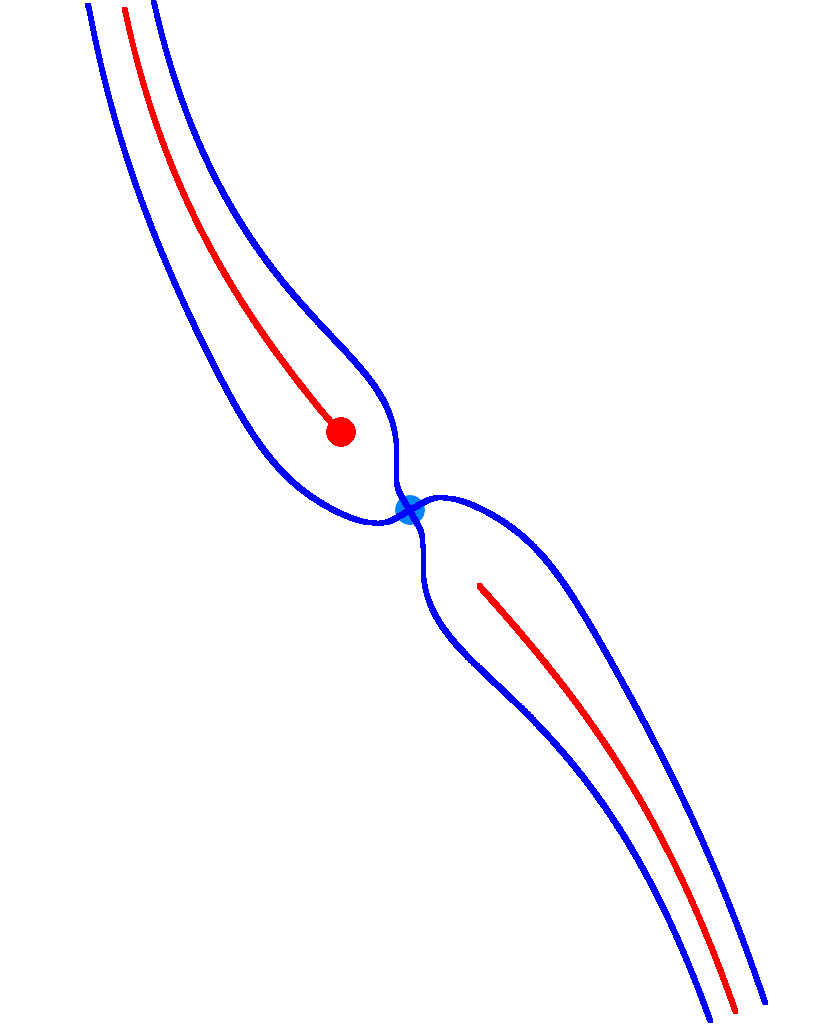}}}
\put(9,0){\fbox{\includegraphics[width=3.6cm]{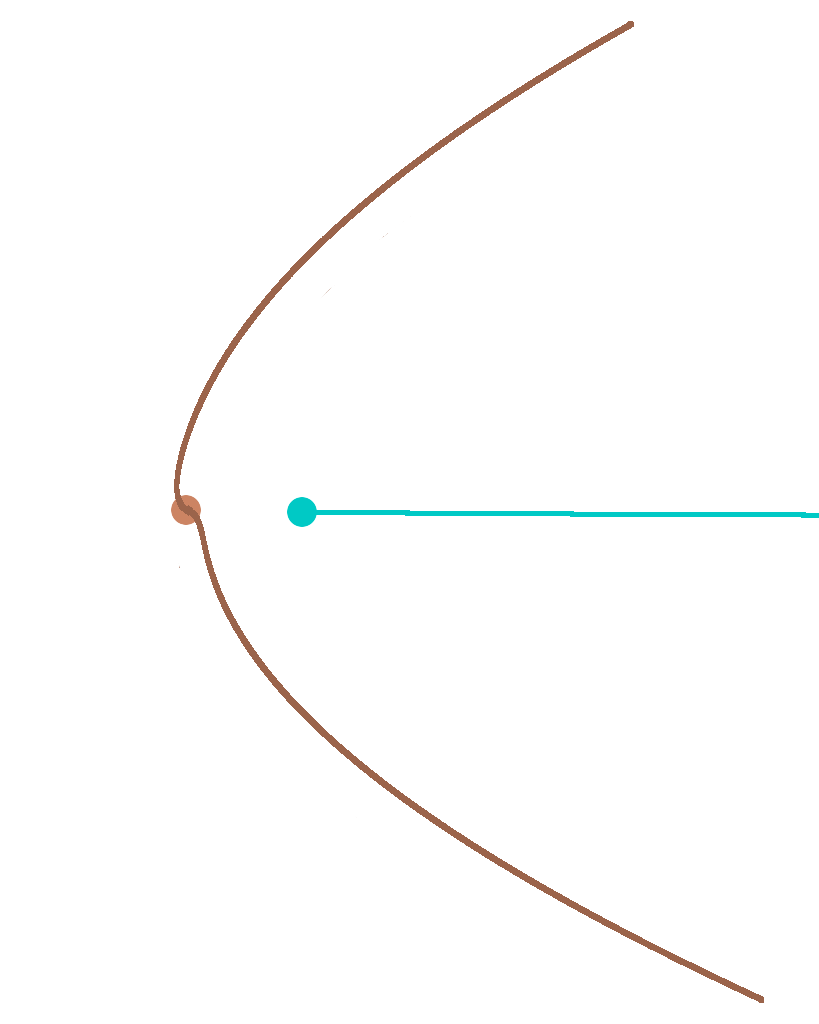}}}
\put(5.,5.5){$f^{\circ (j+k)}$}
\put(9.3,9){$f^{\circ \ell}$}
\put(5.,0.5){$g_m^{\circ (j+k)}$}
\put(9.3,4){$g_m^{\circ \ell}$}
\put(9.04,1.9){$g_m^{\circ \ell}(\omega_m)$}
\put(10.5,2.4){$\alpha$}
\put(10.5,7.4){$\alpha=f^{\circ \ell}(\omega')$}
\put(10.9,4.2){$\eps_m^+$}
\put(11.9,0.5){$\eps_m^-$}
\put(6.4,2.){$\omega_m$}
\put(6.65,2.7){$g_m^{\circ (j+k)}(\zeta_m)$}
\put(6.45,6.85){$\omega'= f^{\circ (j+k)}(\xi)$}
\put(5.4,9.){$\eta_1$}
\put(7.7,5.4){$\eta_2$}
\put(6.,4.1){$\eta_1+\delta_m^-$}
\put(5.,2.4){$\eta_1+\delta_m^+$}
\put(6.8,0.1){$\eta_2+\delta_m^-$}
\put(7.6,2){$\eta_2+\delta_m^+$}
\put(1.9,7.1){$\xi$}
\put(3.8,5.2){$\beta''$}
\put(3.9,9.){$\beta'$}
\put(0.2,7.95){$\beta$}
\put(2.3,2.2){$\zeta_m$}
\put(1.8,1.7){$\xi_m^-$}
\put(2.4,2.7){$\xi_m^+$}
\put(3.7,0.2){$\beta''_m$}
\put(3.9,4.){$\beta'_m$}
\put(0.2,2.95){$\beta_m$}
\put(1.9,5.1){$\theta_1$}
\put(3.9,7.){$\theta_2$}
\put(1.3,9.1){$\theta_3$}
\put(0.2,6.4){$\theta_4$}
\put(2.1,0.1){$\theta_1+\eta_m^+$}
\put(3.4,1.8){$\theta_2+\eta_m^-$}
\put(3.4,2.8){$\theta_2+\eta_m^+$}
\put(1.6,4.1){$\theta_3+\eta_m^-$}
\thicklines
\put(4.5, 0.3){\vector(1,0){1.5}}
\put(8.5, 3.8){\vector(1,0){1.5}}
\put(4.5, 5.3){\vector(1,0){1.5}}
\put(8.5, 8.8){\vector(1,0){1.5}}
\end{picture}
}
\caption{ External rays of $f$ landing near $\xi$, $\omega'$ and $\alpha$ and external rays of $g_m$ landing near $\zeta_m$, $\omega_m$ and $\alpha$. The angles of the external rays are indicated. 
\label{fig:rayspert}}
\end{figure}

Set $\delta_m^\pm := \eps_m^\pm/3^\ell$ and $\eta_m^\pm:= \delta_m^\pm/3^{j+k}$. Then, for $m$ large enough, each of the four rays 
\[R_m(\eta_1+\delta_m^+),\quad R_m(\eta_1+\delta_m^-), \quad R_m(\eta_2+\delta_m^+)\quad\text{and}\quad R_m(\eta_2+\delta_m^-)\]
land in $D'_m$ at a point $z$ satisfying $g_m^{\circ \ell}(z) = g_m^{\circ \ell}(\omega_m)$. This point is necessarily $\omega_m$ itself (see Figure~\ref{fig:rayspert}). Similarly, each of the eight rays 
\[R_m(\theta_1+\eta_m^\pm),\quad R_m(\theta_2+\eta_m^\pm), \quad R_m(\theta_3+\eta_m^\pm)\quad\text{and}\quad R_m(\theta_4+\eta_m^\pm)\]
land in $D''_m$ at a point $z$ satisfying $g_m^{\circ (j+k)}(z) = \omega_m$. This point in necessarily $\xi_m^+$ or $\xi_m^-$. So, four of those rays land at $\xi_m^+$ and four of them land at $\xi_m^-$. 

If $m$ is large enough, $\theta_1+\eta_m^+$, $\theta_2+\eta_m^\pm$ and $\theta_3+\eta_m^-$ belong to $(\theta_1,\theta_3)_\T$ and $\theta_3+\eta_m^+$, $\theta_4+\eta_m^\pm$ and $\theta_1+\eta_m^-$ belong to $(\theta_3,\theta_1)_\T$. 
The rays of angles $\theta_1$ and $\theta_3$ separate the plane in two connected components. 
So, relabelling the points $\xi_m^\pm$ if necessary, we may assume that 
\[R_m(\theta_1+\eta_m^+),\quad R_m(\theta_2+\eta_m^-), \quad R_m(\theta_2+\eta_m^+)\quad\text{and}\quad R_m(\theta_3+\eta_m^-)\]
land at $\xi_m^+$ and that 
\[R_m(\theta_3+\eta_m^+),\quad R_m(\theta_4+\eta_m^-), \quad R_m(\theta_4+\eta_m^+)\quad\text{and}\quad R_m(\theta_1+\eta_m^-)\]
land at $\xi_m^-$ (see Figure~\ref{fig:rayspert}). 

Finally, if $m$ is large enough, one of the angles $1/2$, $1/6$ and $-1/6$ belongs to $(\theta_1+\eta_m^+,\theta_2+\eta_m^-)_\T$, one belongs to $(\theta_2+\eta_m^+,\theta_3+\eta_m^-)_\T$ and one belongs to $(\theta_3+\eta_m^+,\theta_1+\eta_m^-)_\T$. Then, the branching point separating $\beta_m$, $\beta'_m$ and $\beta''_m$ in $\mj_{g_m}$ is $\xi_m=\xi_m^+$. 
\end{proof}

\bibliography{bibliografia}
\bibliographystyle{plain}

\end{document}

%% file: preimalpha.pdf_tex
\begingroup%
  \makeatletter%
  \providecommand\color[2][]{%
    \errmessage{(Inkscape) Color is used for the text in Inkscape, but the package 'color.sty' is not loaded}%
    \renewcommand\color[2][]{}%
  }%
  \providecommand\transparent[1]{%
    \errmessage{(Inkscape) Transparency is used (non-zero) for the text in Inkscape, but the package 'transparent.sty' is not loaded}%
    \renewcommand\transparent[1]{}%
  }%
  \providecommand\rotatebox[2]{#2}%
  \ifx\svgwidth\undefined%
    \setlength{\unitlength}{2047.87814189bp}%
    \ifx\svgscale\undefined%
      \relax%
    \else%
      \setlength{\unitlength}{\unitlength * \real{\svgscale}}%
    \fi%
  \else%
    \setlength{\unitlength}{\svgwidth}%
  \fi%
  \global\let\svgwidth\undefined%
  \global\let\svgscale\undefined%
  \makeatother%
  \begin{picture}(1,0.17134002)%
    \put(0,0){\includegraphics[width=\unitlength]{preimalpha.pdf}}%
    \put(0.77076621,0.13346195){\color[rgb]{0,0,0}\makebox(0,0)[lb]{\smash{$2-1$}}}%
    \put(0.50579509,0.15511503){\color[rgb]{0,0,0}\makebox(0,0)[lb]{\smash{$2-1$}}}%
    \put(0.20488433,0.15946797){\color[rgb]{0,0,0}\makebox(0,0)[lb]{\smash{$1-1$}}}%
    \put(0.78728502,0.07854798){\color[rgb]{0,0,0}\makebox(0,0)[lb]{\smash{$f^{\circ \ell}$}}}%
    \put(0.51762615,0.09316937){\color[rgb]{0,0,0}\makebox(0,0)[lb]{\smash{$f^{\circ k}$}}}%
    \put(0.21124631,0.1072327){\color[rgb]{0,0,0}\makebox(0,0)[lb]{\smash{$f^{\circ j}$}}}%
    \put(0.90559569,0.13346195){\color[rgb]{0,0,0}\makebox(0,0)[lb]{\smash{$D$}}}%
    \put(0.92267258,0.09249971){\color[rgb]{0,0,0}\makebox(0,0)[lb]{\smash{$\alpha$}}}%
    \put(0.66897444,0.1055585){\color[rgb]{0,0,0}\makebox(0,0)[lb]{\smash{$\omega'$}}}%
    \put(0.38324316,0.10667465){\color[rgb]{0,0,0}\makebox(0,0)[lb]{\smash{$\omega$}}}%
    \put(0.08746662,0.10332622){\color[rgb]{0,0,0}\makebox(0,0)[lb]{\smash{$\xi$}}}%
  \end{picture}%
\endgroup%

%% file: preimalphabif.pdf_tex
\begingroup%
  \makeatletter%
  \providecommand\color[2][]{%
    \errmessage{(Inkscape) Color is used for the text in Inkscape, but the package 'color.sty' is not loaded}%
    \renewcommand\color[2][]{}%
  }%
  \providecommand\transparent[1]{%
    \errmessage{(Inkscape) Transparency is used (non-zero) for the text in Inkscape, but the package 'transparent.sty' is not loaded}%
    \renewcommand\transparent[1]{}%
  }%
  \providecommand\rotatebox[2]{#2}%
  \ifx\svgwidth\undefined%
    \setlength{\unitlength}{2047.87814189bp}%
    \ifx\svgscale\undefined%
      \relax%
    \else%
      \setlength{\unitlength}{\unitlength * \real{\svgscale}}%
    \fi%
  \else%
    \setlength{\unitlength}{\svgwidth}%
  \fi%
  \global\let\svgwidth\undefined%
  \global\let\svgscale\undefined%
  \makeatother%
  \begin{picture}(1,0.17941236)%
    \put(0,0){\includegraphics[width=\unitlength]{preimalphabif.pdf}}%
    \put(0.77076621,0.1415343){\color[rgb]{0,0,0}\makebox(0,0)[lb]{\smash{$2-1$}}}%
    \put(0.50579509,0.16318738){\color[rgb]{0,0,0}\makebox(0,0)[lb]{\smash{$2-1$}}}%
    \put(0.20488433,0.16754032){\color[rgb]{0,0,0}\makebox(0,0)[lb]{\smash{$1-1$}}}%
    \put(0.78728502,0.08662032){\color[rgb]{0,0,0}\makebox(0,0)[lb]{\smash{$g_m^{\circ \ell}$}}}%
    \put(0.51215707,0.11217986){\color[rgb]{0,0,0}\makebox(0,0)[lb]{\smash{$g_m^{\circ k}$}}}%
    \put(0.21124631,0.11530505){\color[rgb]{0,0,0}\makebox(0,0)[lb]{\smash{$g_m^{\circ j}$}}}%
    \put(0.90559569,0.1415343){\color[rgb]{0,0,0}\makebox(0,0)[lb]{\smash{$D$}}}%
    \put(0.92267258,0.10057205){\color[rgb]{0,0,0}\makebox(0,0)[lb]{\smash{$\alpha$}}}%
    \put(0.67701063,0.05425231){\color[rgb]{0,0,0}\makebox(0,0)[lb]{\smash{$\omega$}}}%
    \put(0.39775295,0.09465651){\color[rgb]{0,0,0}\makebox(0,0)[lb]{\smash{$\omega'$}}}%
    \put(0.07742139,0.13148904){\color[rgb]{0,0,0}\makebox(0,0)[lb]{\smash{$\xi_m^+$}}}%
    \put(0.88305448,0.03003849){\color[rgb]{0,0,0}\makebox(0,0)[lb]{\smash{$g_m^{\circ \ell}(\omega)$}}}%
    \put(0.63359626,0.13022718){\color[rgb]{0,0,0}\makebox(0,0)[lb]{\smash{$g_m^{\circ k}(\omega')$}}}%
    \put(0.14885422,0.02880438){\color[rgb]{0,0,0}\makebox(0,0)[lb]{\smash{$\xi_m^-$}}}%
    \put(0.13099601,0.0913081){\color[rgb]{0,0,0}\makebox(0,0)[lb]{\smash{$\zeta_m$}}}%
    \put(0.37509535,0.12802904){\color[rgb]{0,0,0}\makebox(0,0)[lb]{\smash{$g_m^{\circ j}(\xi_m^+)$}}}%
    \put(0.35455841,0.05335943){\color[rgb]{0,0,0}\makebox(0,0)[lb]{\smash{$g_m^{\circ j}(\xi_m^-)$}}}%
    \put(0.59977391,0.06898536){\color[rgb]{0,0,0}\makebox(0,0)[lb]{\smash{$D'_m$}}}%
    \put(0.0271952,0.06228852){\color[rgb]{0,0,0}\makebox(0,0)[lb]{\smash{$D''_m$}}}%
  \end{picture}%
\endgroup%